\documentclass[a4paper,10pt]{article}

\usepackage[utf8]{inputenc}
\usepackage[english]{babel}
\usepackage[protrusion=true,expansion=true]{microtype}	
\usepackage{stmaryrd}
\SetSymbolFont{stmry}{bold}{U}{stmry}{m}{n}
\usepackage{amsmath,amsfonts,amsthm,amssymb}
\usepackage[hmargin={30mm,30mm},vmargin={35mm,30mm}]{geometry}
\usepackage{authblk}
\usepackage{caption}

\usepackage{mathpazo}
\usepackage[mathscr]{euscript}

\usepackage{cancel,manfnt,mathtools}

\usepackage{dsfont}
\usepackage{bm}
\usepackage{mathrsfs}
\usepackage{tikz}
\usepackage[bookmarksnumbered,linktoc=all]{hyperref}
\usepackage{tikz-cd} 
\usepackage{graphicx}
\usepackage[]{algorithm2e}
\usepackage{subfig}
\usepackage{wrapfig}
\usepackage{xcolor}
\usepackage{comment}
\usepackage{enumerate}
\usepackage{relsize}
\usepackage{xfrac}
\usepackage{faktor}
\usepackage{ifpdf}
\ifpdf
\fi

\numberwithin{equation}{section}

\newtheorem{theorem}{Theorem}[section]

\newtheorem{lemma}[theorem]{Lemma}
\newtheorem{proposition}[theorem]{Proposition}
\newtheorem{definition}[theorem]{Definition}
\newtheorem{problem}[theorem]{Problem}

\newtheorem{conjecture}[theorem]{Conjecture}

\theoremstyle{remark}
\newtheorem{remark}[theorem]{Remark}
\newtheorem{example}[theorem]{Example}

\newcommand{\C}{\mathbb C}
\newcommand{\Z}{\mathbb Z}

\newcommand{\PP}{\mathbb P}
\newcommand{\A}{\mathbb A}

\newcommand{\F}{\mathcal{F}}
\newcommand{\OO}{\mathcal{O}}

\newcommand{\Ee}{\mathcal{E}}

\newcommand{\ZZ}{\mathcal{Z}}
\newcommand{\ttt}{\mathfrak{t}}

\newcommand{\w}{w}
\newcommand{\gl}{\mathfrak{gl}}
\newcommand{\ssl}{\mathfrak{sl}}

\newcommand{\geg}{\mathfrak{g}}
\newcommand{\kek}{\mathfrak{k}}

\newcommand{\Geg}{\mathcal{G}}

\newcommand{\bb}{\mathfrak b}

\newcommand{\I}{\mathcal I}
\newcommand{\Ss}{\mathcal S}

\newcommand{\Gs}{\mathrm G}
\newcommand{\Ts}{\mathrm T}
\newcommand{\Ks}{\mathrm K}
\newcommand{\Bs}{\mathrm B}
\newcommand{\Ps}{\mathrm P}
\newcommand{\Ws}{\mathrm W}

\newcommand{\rk}{\operatorname{rk}}
\newcommand{\Cs}{\mathbb C^\times}
\newcommand{\Vect}{\mathrm{Vect}}
\renewcommand{\bf}{\bfseries}

\newcommand{\Tr}{\operatorname{Tr}}
\newcommand{\diag}{\operatorname{diag}}

\newcommand{\Spec}{\operatorname{Spec}}

\newcommand{\pt}{\operatorname{pt}}
\newcommand{\ch}{\operatorname{ch}}

\newcommand{\id}{\operatorname{id}}

\newcommand{\Ad}{\operatorname{Ad}}
\newcommand{\Lie}{\operatorname{Lie}}

\newcommand{\GL}{\operatorname{GL}}
\newcommand{\SL}{\operatorname{SL}}

\newcommand{\PGL}{\operatorname{PGL}}

\newcommand{\SO}{\operatorname{SO}}

\newcommand{\Span}{\operatorname{span}}

\newcommand{\Stab}{\operatorname{Stab}}
\newcommand{\reg}{\mathrm{reg}}
\newcommand{\tot}{\mathrm{tot}}
\newcommand{\Gr}{\operatorname{Gr}}

\newcommand{\Fix}{\operatorname{Fix}}

\newcommand{\red}{\operatorname{red}}

\makeatletter
\usepackage{fancyhdr}
\makeatother

\title{Variations on cohomology rings and zero schemes}
\author{
Kamil Rychlewicz\footnote{École Polytechnique Fédérale de Lausanne, {\tt kamil.rychlewicz@epfl.ch}}
}
\date{}
\begin{document}

\maketitle 

\begin{abstract}
We extend the theorem of Hausel and the author \cite{HR} that relates equivariant cohomology rings and algebras of functions on zero schemes. This paper combines three separate results:
\begin{enumerate}
    \item We prove that for a reductive group $\Gs$ acting on a smooth projective variety one can see the equivariant cohomology ring as the ring of functions on the zero scheme over the Kostant section, provided that some transversality condition is satisfied. In particular, we show that the conclusion holds for spherical varieties;
    \item We show a version for singular varieties, where in general we only recover a part of equivariant cohomology ring, generated by Chern classes;
    \item We show that an analogous result, connecting equivariant K-theory to the ring of functions on the fixed-point scheme, holds for GKM spaces.    
\end{enumerate}
\end{abstract}

\section{Introduction}
The classical Poincar\'{e}--Hopf theorem allows one to relate the Euler characteristic, a topological invariant of a manifold, to local data extracted from the zeros of a vector field. As the seminal work of Carrell and Lieberman shows, one can extract more information if the manifold is a complex algebraic variety. Precisely, one has the following theorem \cite{CL}:

\begin{theorem}
 Let $X$ be a smooth projective complex algebraic variety of dimension $n$. Consider an algebraic vector field $V\in\Vect(X)$ and assume that its zero set is isolated, but nonempty. Let $Z$ denote the zero scheme of $V$. Then the coordinate ring $\C[Z]$ admits an increasing filtration 
 $$F_\bullet: 0 = F_{-1} \subset F_0 \subset F_1 \subset \dots \subset F_n = \C[Z]$$
 such that there is an isomorphism
 $$H^*(X,\C) \simeq \Gr_F(\C[Z])$$
 of graded algebras. The degree on the left is twice the degree on the right, and in particular the odd cohomology of $X$ vanishes.
\end{theorem}

In principle determining the filtration $F_\bullet$ is a difficult task. However, one can show \cite{AC,ACLS} that in certain situations it is induced by a grading, and hence $H^*(X,\C)\simeq \C[Z]$. This is the case when the Borel subgroup of triangular matrices in $\SL_2$ acts on $X$, and the vector field $V$ is associated to the element $e = \left(\begin{smallmatrix} 0 & 1\\ 0 & 0 \end{smallmatrix}\right) \in \ssl_2$. After \cite{BC} we say in that case that the group acts \emph{regularly}. By a result of Horrocks \cite{Horrocks}, if the vector field $V$ associated to an action of the additive group has isolated zeros, it has in fact a unique zero. This seemingly strong condition holds for actions of principal $\SL_2$ subgroups on many well-studied spaces, including flag varieties, smooth Schubert varieties and Bott--Samelson resolutions.

In \cite{BC} Brion and Carrell proved that in a similar way one can recover $\Cs$-equivariant cohomology rings. One has to deform the zero scheme, to obtain a scheme over $\Spec H^*_{\Cs}(\pt) = \A^1$.

\begin{theorem}
Assume that $\Bs_2$, the Borel subgroup of upper triangular matrices in $\SL_2$, acts regularly on a smooth projective variety $X$. Let $\Ss = e + \ttt \subset \bb_2$, where $\ttt$ is the diagonal Cartan subalgebra of $\ssl_2$. Let $\ZZ$ be the zero scheme of the vector field defined on $\Ss\times X$ by the infinitesimal action of the Lie algebra. Then $\ZZ \simeq \Spec H^*_{\Cs}(X)$, where $H^*_{\Cs}(X)$ is the cohomology ring of $X$ equivariant with respect to one-dimensional diagonal torus in $\Bs_2$. Moreover, the structure of algebra over $H^*_{\Cs}(\pt)\simeq \C[\ttt]$ is given by the projection map $\ZZ\to \Ss\simeq \ttt$. The grading on $H^*_{\Cs}$ is determined by the weights of the $\Cs$ representation on $\C[\ZZ]$. It comes from the action of $\Cs$ on $\ZZ\subset \Ss\times X$ given by
$$t.(e+v,x) = (e+t^{-2}v,H^tx),$$
where for $t\in\Cs$ we denote $H^t = \left(\begin{smallmatrix} t & 0\\ 0 & t^{-1} \end{smallmatrix}\right)$.
\end{theorem}

In \cite{HR} Hausel and the author extend the result to a much wider class of algebraic groups, called \emph{principally paired}. The class includes all connected reductive groups and their parabolic subgroups. The affine plane $\Ss$ in general has to be replaced with a Kostant section inside the Lie algebra $\geg$ of the acting group $\Gs$. The torus action is analogous, and one gets the affine zero scheme $\ZZ$ over $\Spec H^*_{\Gs}(\pt)$, and hence $\C[\ZZ]$ is a graded algebra over $H^*_{\Gs}(\pt)$, proved to be canonically isomorphic to $H^*_{\Gs}(X)$. As before, the assumption is that $X$ is smooth projective, and the regular nilpotent has a unique zero. A version for singular varieties is also proved based on an analogous result of \cite{BC}. The assumptions require the singular variety to be embedded in a smooth variety with regular action, and the restriction on cohomology rings has to be surjective.

This paper can be viewed as a continuation of \cite{HR}. It presents a three-fold generalisation of the results.

First, in Section \ref{secnon}, we show that the regularity assumption can be relaxed. We only ask for the zero scheme $\ZZ$ to be of the same dimension as the base $\Ss$, i.e equal to the rank of $\Gs$. We prove in Theorem \ref{nonreg} that the ring of functions on $\ZZ$ is still canonically isomorphic to the equivariant cohomology ring, as a graded algebra over $H^*_{\Gs}(\pt)$. However, the cost of this generalisation is losing the affineness of $\ZZ$, which only holds true if the action is regular. It then poses computational difficulties to potential calculations on cohomology rings. The examples are however plentiful, as we show that every spherical variety satisfies the conditions.

Second, in Section \ref{secsing} we extend the singular case. We consider a smooth projective ambient variety $X$ with a regular $\Gs$-action. As above, we construct the zero scheme $\ZZ^X\subset \Ss\times X$, which we know is isomorphic to $\Spec H^*_{\Gs}(X)$. We then take any singular $\Gs$-equivariant projective subvariety $Y\subset X$ and consider the reduced intersection $\ZZ^Y = (\ZZ^X\cap (\Ss\times Y))^{\red}$. We know by \cite{BC,HR} that $\ZZ^Y \simeq \Spec H^*_{\Gs}(Y)$ if the restriction $H^*(X)\to H^*(Y)$ is surjective. In fact, one can easily see that the surjectivity assumption is also necessary. Nevertheless, we show that even without this assumption one can read off a ``geometric'' part of cohomology ring from $\ZZ^Y$. Concretely, we denote by $\widetilde{H}^*_\Gs(Y) \subset H^*_\Gs(Y)$ the subalgebra generated by the Chern classes of $\Gs$-linearised vector bundles on $Y$. We then show in Theorem \ref{singful} that $\ZZ^Y\simeq \Spec \widetilde{H}^*_\Gs(Y)$. This way we get some information on the cohomology ring of e.g. discriminantal varieties, cf. Example \ref{discex}.

Lastly, in Section \ref{seck} we discuss the K-theoretic conjecture of Hausel. We first define the \emph{fixed-point scheme} $\Fix_\Gs(X)$ of an algebraic group $\Gs$ acting on a variety $X$. Its points parametrise the pairs $(g,x)\in \Gs\times X$ such that $gx = x$. Concretely, it is defined via the fibre product\footnote{This definition was provided by Jakub L\"{o}wit \cite{jakub}}

$$ \begin{tikzcd}
  \Fix_\Gs(X) \arrow{r}\arrow{d} &
  \Gs\times X\arrow{d}{\rho\times \pi_2}  \\
  X \arrow{r}{\Delta} &
  X\times X.
\end{tikzcd}$$

\begin{conjecture}
 Assume that a principally paired group $\Gs$ acts on a smooth projective variety regularly. Let $\Gs$ act on $\Fix_\Gs(X)$ by $g\cdot (h,x) = (ghg^{-1},gx)$. Then the ring $\C[\Fix_\Gs(X)]^\Gs$ of $\Gs$-invariant functions on $\Fix_\Gs(X)$ is an algebra over $\C[\Gs]^\Gs\cong K^0_{\Gs}(\pt)\otimes \C$ isomorphic to the complexified equivariant algebraic K-theory $K^0_{\Gs}(X)\otimes \C$.
\end{conjecture}

In Theorem \ref{gkmk} we prove that the statement holds true for smooth projective GKM varieties, i.e. varieties equipped with a torus action with finitely many 0- and 1-dimensional orbits. Notice that a torus action is never regular on a variety of positive dimension. Therefore our result does not directly constitute a case of the conjecture.

In the same spirit, the results on equivariant cohomology in \cite{HR} were shown to hold true both for regular varieties and for GKM spaces, but without a clear connection between the two cases. However, we hope that the two results have a common generalisation into one framework, both in cohomology and in K-theory case.

The results of this paper form a part of the author's PhD thesis \cite{Rych}. For more details on the theory, and a broader historical account, we invite the reader to consult the thesis itself.

\subsection{Notation}
All the algebraic schemes in this paper, and in particular all the algebraic groups, are defined over the complex numbers. Therefore for any scheme $X$, by $\C[X] = \Gamma(X,\OO_X)$ we denote the ring of global functions on $X$.

We frequently use linear groups $\SL_n$, $\GL_n$, and their Lie algebras $\ssl_n$, $\gl_n$. We use the standard notation for the basis of $\ssl_2$:
$$
e=\begin{pmatrix}
    0 & 1 \\
    0 & 0
\end{pmatrix}, \qquad
f=\begin{pmatrix}
    0 & 0 \\
    1 & 0
\end{pmatrix}, \qquad
h=\begin{pmatrix}
    1 & 0 \\
    0 & -1
\end{pmatrix}.
$$

Additionally, by $\Bs_n\subset \SL_n$ we denote the Borel subgroup of upper triangular matrices of determinant $1$, and by $\bb_n$ we denote the corresponding Lie algebra. In particular, as a vector space, $\bb_2$ is spanned by $e$ and $h$.

For a linear algebraic group $\Gs$ with the Lie algebra $\geg$, we call an element $v\in\geg$ regular, if its centraliser $C_\geg(v)$ in $\geg$ has the dimension equal to the dimension of a maximal torus $\Ts$ in $\Gs$. If a regular element exists, the Lie algebra $\ttt$ of $\Ts$ is a Cartan subalgebra, and its dimension is the rank of the group. The adjoint representation of $\Gs$ restricted to $\Ts$ splits into weight spaces. The weights appearing are elements of $\ttt^*$. The regular elements of $\ttt$ are then those that are not annihilated by any of the nonzero weights thereamong. Recall that for $\Gs$ reductive those are called the roots of the group.

Whenever we talk about singular cohomology groups, those are meant to have complex coefficients. For an algebraic group $\Gs$ and a $\Gs$-variety $X$, we also consider the $\Gs$-equivariant cohomology $H^*_\Gs(X) = H^*(\mathrm{E}\Gs\times_\Gs X, \C)$, defined as the cohomology of the mixed product $\mathrm{E}\Gs\times_\Gs X$. The ($\Gs$-equivariant) K-theory ring $K^0(X)$ ($K^0_\Gs(X)$) of a ($\Gs$-)variety is defined as the Grothendieck group of ($\Gs$-linearised) vector bundles on $X$, with relations $[\Ee] = [\Ee_1]+[\Ee_2]$ for any exact sequence of vector bundles $0\to\Ee_1\to\Ee\to\Ee_2\to 0$. It admits a multiplicative structure given by the tensor product of bundles.

For any $\Z_{\ge 0}$-graded $\C$-algebra $R$ such that every graded piece $R_i$ is finite-dimensional, we define the Poincar\'{e} series of $R$ as
$$P_R(t) = \sum_{i=0}^\infty \dim_\C(R_i) t^i \in \Z \llbracket t \rrbracket.$$

\subsection{Acknowledgements}
This work was a part of the author's work during the PhD studies in Institute of Science and Technology Austria. The work was partially supported by the DOC Fellowship of the Austrian Academy of Sciences \emph{Topology of open smooth varieties with a torus action.}

The author would like to express his gratitude to his PhD advisor, Tam\'{a}s Hausel, for the guidance and useful suggestions, as well as to Michel Brion, Mischa Elkner, Daniel Holmes, Andreas Kretschmer, Jakub L\"{o}wit, Leonid Monin, Shon Ngo, Anna Sis\'{a}k and Michael Thaddeus, for helpful comments and remarks.

The figures in this paper were produced using Wolfram Mathematica software.

\section{Regular actions and zero schemes}

\subsection{A short recap on Kostant sections and equivariant cohomology} \label{recap}
We first recall the definition of a \emph{principally paired group} \cite[Definition 2.20]{HR}.

\begin{definition}
Assume that a connected linear algebraic group $\Gs$ is given, together with a non-trivial homomorphism $\phi:\Bs_2\to\Gs$. We call $\Gs$ \emph{principally paired} if $d\phi(e)$ is a regular element of $\geg$.
\end{definition}

The most important examples of principally paired groups are reductive groups and their parabolic subgroups. For a principally paired group like above, we will abuse the notation and denote $d\phi(e)$ by $e$. The maximal, diagonal torus of $\Bs_2$ is sent by $\phi$ to a one-dimensional torus in $\Gs$. This integrates the image of $h$ along $d\phi$. We will hence denote 
$$H^t = \phi\left(\left(\begin{smallmatrix} t & 0\\ 0 & t^{-1} \end{smallmatrix}\right)\right) \text{ for }t\in \Cs.$$

Notice that the non-triviality of $\phi$ is simply equivalent to the non-triviality of $d\phi(h)$. However, a group can still be principally paired with $e$ mapping to zero. This is the case for any torus, where any closed one-parameter subgroup defines a morphism from $\Bs_2$, factoring via the quotient by the unipotent part.

\begin{example} \label{exgl}
Consider a homomorphism of $\bb_2 = \Lie(\Bs_2)$ to $\gl_{n+1}$ given by
$$e\mapsto \begin{pmatrix}
       0 & 1 & 0 & 0 & \dots & 0\\
       0 & 0 & 1 & 0 & \dots & 0\\
       0 & 0 & 0 & 1 & \dots & 0\\
       \vdots & \vdots & \vdots & \vdots & \ddots & \vdots \\
       0 & 0 & 0 & 0 & \dots & 1 \\
       0 & 0 & 0 & 0 & \dots & 0
      \end{pmatrix},
\qquad
h\mapsto
\begin{pmatrix}
       n & 0 & 0 & 0 & \dots & 0\\
       0 & n-2 & 0 & 0 & \dots & 0\\
       0 & 0 & n-4 & 0 & \dots & 0\\
       \vdots & \vdots & \vdots & \vdots & \ddots & \vdots \\
       0 & 0 & 0 & 0 & \dots & 0 \\
       0 & 0 & 0 & 0 & \dots & -n
\end{pmatrix}.
$$
As the image of $h$ lies in the Lie algebra of the maximal torus, this clearly integrates to a map $\phi:\Bs_2\to \GL_{n+1}$  of algebraic groups. Moreover, notice that the image of $e$ is regular in $\gl_{n+1}$. Therefore, $\GL_{n+1}$ is principally paired. In fact, the image of $\phi$ lies in the Borel subgroup of upper-triangular matrices. This demonstrates that all parabolic subgroups of $\GL_{n+1}$ are principally paired as well. 
\end{example}

For principally paired groups one can then construct an equivalent of the Kostant section $\Ss$, as shown in \cite{Kostsec} for reductive groups. The detailed construction can be found in \cite[2.6.3]{HR}. It depends on a few choices within the Lie algebra, including the choice of the principal pairing homomorphism $\phi:\Bs_2\to \Gs$. We note here the important properties and examples of $\Ss$. Denote by $\Ts$ the maximal torus in $\Gs$ containing $\{H^t\}$ and let $\ttt = \Lie(\Ts)$. We denote by $\Ws = N(\Ts)/\Ts$ the Weyl group of $\Gs$.

\begin{proposition}\label{kostprop}
The Kostant section $\Ss$ defined in \cite[2.6.3]{HR} for a principally paired group $\Gs$ satisfies the following conditions.
\begin{enumerate}
    \item It is an affine subspace of $\geg$, of dimension equal to $r=\rk \Gs$, and containing $e$.
    \item If $\Gs$ is solvable, then $\Ss = e + \ttt$.
    \item If $\Gs$ is reductive, then $\Ss = e + C_\geg(f)$, where $(e,f,h)$ is an $\SL_2$-triple.
    \item The Kostant section parametrises the regular conjugacy types. i.e. all the elements of $\Ss$ are regular and every regular conjugacy orbit intersects $\Ss$ in exactly one point.
    \item The restriction of functions from $\geg$ to $\Ss$ is an isomorphism when restricted to the $\Gs$-invariant functions, i.e. 
    $$\C[\Ss] \cong \C[\geg]^\Gs.$$
    Note that by Chevalley's restriction theorem \cite[Theorem 3.1.38]{ChGi} the right-hand side also restricts isomorphically to $\C[\ttt]^\Ws$.
\end{enumerate}
\end{proposition}
\begin{proof}
    See Section 2.6 in \cite{HR}.
\end{proof}
\noindent
Notice that $\C[\ttt]^\Ws \cong H^*_\Gs(\pt)$, and hence $\Gs$-equivariant cohomology is always a module over $\C[\ttt]^\Ws \cong \C[\Ss]$.

Whenever an algebraic group $\Gs$ acts on a smooth variety $X$, it yields an infinitesimal action of the Lie algebra $\geg$ of $X$. Formally, the action map $\Gs\times X\to X$ yields the map $T\Gs\times TX\to TX$ of the total spaces of the tangent bundles. Restricting to $T_1\Gs\times X \simeq \geg\times X$ gives the Lie algebra action morphism $\geg\times X \to TX$.

Extending this by the identity on $\geg$ yields $\geg\times X\to \geg\times TX$, which is a section of the vertical tangent bundle on $\geg\times X$, i.e. the tangent bundle $\geg\times TX$ in the direction of $X$. We call this vector field $V$ the \emph{total vector field} of the action\footnote{The author would like to thank Jakub L\"{o}wit for suggesting this slick definition. The definition in \cite{HR} is technically more involved.}. For any $y\in\geg$ this restricts to a vector field on $X$, which we denote by $V_y$. We also consider the restriction $V_\Ss$ of $V$ to $\Ss\times X$.

We then construct two zero schemes: $\ZZ_\tot\subset \geg\times X$ is the zero scheme of $V$, and $\ZZ\subset \Ss\times X$ is the zero scheme of $V_\Ss$. Equivalently, $\ZZ$ is the pullback of $\ZZ_\tot$ via the map $\Ss\times X\to \geg\times X$. We also denote by $\ZZ_{\tot}^{\red}\subset \ZZ_{\tot}$ the closed subscheme of $\ZZ_{\tot}$ obtained by reduction. Explicitly, the scheme $\ZZ_\tot$ can be defined in two ways. First, we can describe its ideal sheaf in $\geg\times X$ as the image of the contraction of vertical 1-forms with $V$. Alternatively, it fits into the following diagram as the scheme-theoretic intersection of $V$ and the zero section.

$$ \begin{tikzcd}
 \ZZ_\tot  \arrow{d}{} \arrow{r}{}  &
   X \arrow{d}{0} \\
   \geg\times X \arrow{r}{V}&
  TX
\end{tikzcd}$$

There is a natural $\Cs$-action on $\ZZ_\tot$, which simply comes from scaling the $\geg$ component. For computational reasons we consider the action of the weight $-2$, i.e. $t . v  = t^{-2}v$.

The action on $\ZZ_\tot$ does not directly restrict to an action on $\ZZ$, as scaling does not preserve $\Ss$. Indeed, $\Ss$ is only an affine space in $\geg$, not a linear subspace. However, if $v\in \geg$ is regular, then for any $t\in\Cs$ the element $t^{-2}v$ is also regular, hence it must be conjugate to an element of $\Ss$. Specifically, for $v\in\Ss$ we have $\Ad_{H^t}(t^{-2}v)\in \Ss$. Hence we can define the $\Cs$-action on $\Ss$ by
\begin{equation}
\label{defact}
    t.v = t^{-2}\Ad_{H^t}(v).
\end{equation}
With that action, the isomorphism
$$\C[\Ss]\simeq H^*_\Gs(\pt)$$
is a graded isomorphism. The grading on the right is the standard cohomological grading, and on the left-hand side it comes from the weights of pullback by the action of $\Cs$. For the details, see \cite[2.6.3]{HR}. The action \ref{defact} naturally defines a $\Cs$ action on $\Ss\times X$ by
$$t.(v,x) = \left( t^{-2}\Ad_{H^t}(v), H^t x\right). $$
Note that $V_{ t^{-2}\Ad_{H^t}(v)}(H^t x) =t^{-2}H^t_*V_v(x) $. Hence the $\Cs$-action defined on $\Ss\times X$ satisfies
$$t_*V = t^2V$$
for any $t\in\Cs$. Thus $\ZZ$ is $\Cs$-invariant, as the zero scheme of $V$. The pullback of global functions via the action makes $\C[\ZZ]$ into a $\Cs$-representation, and hence a ring graded by the weights of $\Cs$:
$$\C[\ZZ] = \bigoplus_{k\in\Z} \C[\ZZ]_k.$$

If a principally paired group $\Gs$ acts on an algebraic variety $X$, we say that the action is \emph{regular} if the principal nilpotent $e\in\geg$ has a single zero on $X$. On one hand, this is equivalent to a seemingly weaker condition of $e$ having isolated zeros. On the other hand, it is equivalent to a seemingly stronger condition of all regular elements of $\geg$ having isolated zeros \cite[Lemma 2.46]{HR}.

The following theorem is then proved in \cite[Theorems 3.5, 4.1, 5.7]{HR}.

\begin{theorem}\label{thmhr}
Assume that a principally paired group $\Gs$ acts regularly on a smooth projective variety $X$. Then there are two isomorphisms of graded $\C[\ttt]^\Ws$-algebras:
$$H_\Gs^*(X) \to H^0(\ZZ_{\tot}^{\red},\OO_{\ZZ_{\tot}^{\red}})^\Gs \to H^0(\ZZ,\OO_\ZZ)$$
between the $\Gs$-equivariant cohomology of $X$, the ring of $\Gs$-invariant functions on $\ZZ_\tot^{\red}$, and the ring of all regular functions on $\ZZ$. The latter isomorphism comes from the restriction.
Moreover, the zero scheme $\ZZ$ is reduced and affine, so that we have the following diagram with the vertical maps being isomorphisms:
$$ \begin{tikzcd}
 \ZZ  \arrow{d}{\pi} \arrow[r,"\rho^*", "\cong"']  &
  \Spec H_\Gs^*(X;\C) \arrow{d} \\
   \Ss \arrow{r}{\cong}&
  \Spec H^*_\Gs.
\end{tikzcd}$$
\end{theorem}
\begin{remark}
 It is not known to the author whether $\ZZ_\tot$ itself has to be reduced under the conditions of the theorem. We address this issue again in Sections \ref{secnon} and \ref{seck}.
\end{remark}

At the heart of \cite{HR} lie the non-equivariant results of Carrell--Lieberman and Akyildiz--Carrell \cite{AC, CL}. In the generality we need, \cite[Theorem 1.1]{AC} says the following.

\begin{theorem}\label{clthm}
 Assume that the Borel subgroup $\Bs_2$ of upper-triangular matrices in $\SL_2$ acts on a smooth projective variety $X$ regularly, i.e. with a single zero of $e\in\ssl_2$. Let $V$ be the vector field defined by $e$ on $X$. Then the action of the diagonal torus $\Cs\subset \SL_2$ preserves the zero scheme $Z$ of $V$ and all the weights of its action on $\C[Z]$ are even. This gives a grading $\C[Z] = \bigoplus_{i=0}^n \C[Z]_{2i}$, where $\C[Z]_{2i}$ is the part of $\C[Z]$ of weight $2i$. Then
 $$H^*(X,\C) \simeq \C[Z]$$
 as graded rings.
\end{theorem}

Therefore $Z \simeq \Spec H^*(X,\C)$ and the cohomological grading is recovered geometrically by the weights of the $\Cs$-action. To arrive at Theorem \ref{thmhr}, one first proves that in case $\Gs$ is solvable, $\ZZ$ is a reduced affine scheme, and defines a map $H^*_\Gs(X,\C)\to \C[\ZZ]$ by defining the values on the closed points of $\ZZ$. The injectivity of localisation shows that the map is injective, and Theorem \ref{clthm} provides a comparison of the Poincar\'{e} series. Then, in the case of a general group $\Gs$, one realises the Weyl group action on $H^*_\Gs(X,\C)$ as a geometric action on the zero scheme defined for the Borel subgroup, to get the result for arbitrary principally paired $\Gs$.

In the next section we prove that the isomorphism
$$H_\Gs^*(X) \to  H^0(\ZZ,\OO_\ZZ)$$
holds with assumptions weaker than regularity. In that case we will not be able to use the results of \cite{AC,CL} directly, as the assumptions do not hold. Instead, we adjust their proofs to fit our needs. To make the analogy clear, we provide a quick sketch of the proof of the above theorem.

\begin{proof}[Sketch of proof of Theorem \ref{clthm}]
 The vector field $V$ defines a section of the tangent bundle $TX$ of $X$. Consider the Koszul complex defined by that section:
 $$ 0 \to \Omega^n_X \xrightarrow{\iota_V} \Omega^{n-1}_X  \xrightarrow{\iota_V} \dots  \xrightarrow{\iota_V} \Omega^{1}_X  \xrightarrow{\iota_V} \OO_X \to 0. $$
 By the assumption the codimension of $Z$ in $X$ is equal to the rank of $TX$, hence the Koszul complex is a resolution of $\OO_Z$ and thus we obtain the spectral sequence
 $$
	E_1^{pq} = H^q(X,\Omega^{-p})
 $$
convergent to $H^{p+q}(X,\OO_Z)$. 

Using the Lefschetz operator, one then proves that the spectral sequence degenerates on the first page. As $\dim Z = 0$, we have 
$$H^i(X,\OO_Z) =
\begin{cases}
\C[Z] \text{ for } i=0; \\
0 \text{ otherwise.}
\end{cases}
$$
Therefore the convergence means that there is an increasing filtration $F_i$ on the ring $\C[Z]$ such that $F_i/F_{i-1} = H^i(X,\Omega^i)$, and moreover $H^q(X,\Omega^p) = 0$ for $p\neq q$. 

Now we need to see that the filtration is associated with the grading given by the weights of the $\Cs$-action. First, for $t\in \Cs$ let $t_p:\Omega^p_X\to \Omega^p_X$ be the pullback of forms along the map $t:X\to X$. We then construct an action of $\Cs$ on the Koszul resolution via the following.

$$
\begin{tikzcd}
0\arrow[r] &
\Omega^n_X \arrow[rr, "\iota_V"] \arrow[d, "t^{2n} t^{-1}_n"]&&
\Omega^{n-1}_X \arrow[rr, "\iota_V"] \arrow[d, "t^{2(n-1)} t^{-1}_{n-1}"] &&
\dots \arrow[r, "\iota_V"] \arrow[d]&
\Omega^{1}_X \arrow[rr, "\iota_V"] \arrow[d, "t^{2} t^{-1}_1"]  &&
\OO_X \arrow[r] \arrow[d,"(t^{-1})^*"] & 0\\
0\arrow[r] &
\Omega^n_X \arrow[rr, "\iota_V"] &&
\Omega^{n-1}_X \arrow[rr, "\iota_V"] &&
\dots \arrow[r, "\iota_V"] &
\Omega^{1}_X \arrow[rr, "\iota_V"] &&
\OO_X \arrow[r] & 0
\end{tikzcd}
$$ 
The vertical maps commute with the horizontal ones due to the equality
$$t_* V = t^2 V.$$
Hence we get a group of automorphisms of the Koszul complex, parametrised by $\Cs$. The spectral sequences for hypercohomology are functorial and thus the sequence 
$$
	E_1^{pq} = H^q(X,\Omega^{-p}) \Rightarrow H^{p+q}(X,\OO_Z)
$$
is $\Cs$-equivariant. The $\Cs$-action on $\OO_X$, the zeroth term of the Koszul complex, is defined by the geometric action of $\Cs$ on $X$. Hence on the right-hand side we see the action of $\Cs$ on $\OO_Z$ which descends from the action on $X$. We need to determine the action on the left-hand side. By definition, $t$ acts on $\Omega^p$ by $t^{2p} t^{-1}_p$. The map $t^{-1}_p$ is the pullback of $p$-forms via the action of $\Cs$ on $X$. As $\Cs$ is connected, for any $t$ the corresponding map $t^{-1}$ is homotopic to the identity, and hence descends to the identity on the level of $H^q(X,\Omega^p)$. Therefore the action on the left-hand side is simply the multiplication by $t^{2p}$. Hence the $\Cs$-invariant filtration $F_\bullet$ on $\C[Z]$ satisfies the property that the action of $\Cs$ on $F_i/F_{i-1}$ is of pure weight $2i$. But there is only one such filtration, and it is the one defined by the grading by weights.

For details, see \cite[Main Theorem and Remark 2.7]{CL} and \cite[Theorem 1.1]{AC}, or \cite[Section 4]{Rych}.
\end{proof}

\subsection{Zero schemes and equivariant cohomology for non-regular actions}
\label{secnon}
Theorem \ref{thmhr} surprisingly connects two types of commutative rings arising in geometry -- cohomology rings of topological spaces, and rings of functions on algebraic varieties. It is however restricted by its assumptions. The regularity condition is satisfied e.g. for flag varieties, smooth Schubert varieties and Bott--Samelson resolutions. On the other hand, many interesting varieties with a group action do not satisfy the assumption. Even for $\Gs$-spaces with isolated torus-fixed points, the regular nilpotent might have the zero scheme of positive dimension -- as e.g. for the actions of reductive groups on wonderful compactifications. Moreover, if $\Gs = \Ts$ is a torus, the regular nilpotent is equal to $0$, hence in non-trivial cases its zero scheme is also non-trivial.

The aim of this section is to fix that gap and prove the following theorem. The assumption will cover many interesting examples, such as spherical varieties, as one can see in Example \ref{exsph}. This will show the isomorphism of the ring of functions on the zero scheme, and the equivariant cohomology of the given variety. However, in this more general setup one loses the affineness of $\ZZ$.

\begin{theorem}\label{nonreg}
Assume that a principally paired group $\Gs$ acts on a smooth projective variety $X$. Let $\Ss \subset \geg$ be the Kostant section. Let $\ZZ \subset \Ss\times X$ be the zero scheme of the vector field $V_\Ss$. Assume that $\dim \ZZ = \dim \Ss$. Then there is a natural isomorphism
$$\rho: H^*_\Gs(X) \to \C[\ZZ] $$
of graded $\C[\Ss]\simeq H^*_\Gs(\pt)$-algebras. In other words, the following diagram commutes.
$$
\begin{tikzcd}
H^*_\Gs(X) \arrow[r, "\rho", "\cong"'] 
& \C[\ZZ]
\\ \\
H^*_\Gs(\pt) \arrow[r, "\simeq"]\arrow[uu]
& \C[\Ss] \arrow[uu]
\end{tikzcd}
$$
The grading on $\C[\ZZ]$ is defined by the weight spaces of the $\Cs$-representation thereon. It comes from the pullback of functions along the $\Cs$-action on $\ZZ$ given by
$$t\cdot (v,x) = \left(\frac{1}{t^2}\Ad_{H^t}(v), H^t x\right).$$
Moreover, $H^i(\ZZ,\OO_\ZZ) = 0$ for any $i>0$. 

The isomorphism $\rho$ is natural with respect to morphisms of admissible $\Gs$-varieties, as well as for the group homomorphisms which preserve the chosen $\ssl_2$-pairs.
\end{theorem}

\begin{remark}
Note first that any regular action satisfies the $\dim \ZZ = \dim \Ss$ condition. Indeed, by \cite[Lemma 2.46]{HR}, in that case all the fibres of the map $\ZZ \to \Ss$ are finite. 
\end{remark}

We split the proof of the theorem into lemmas. First, we treat the case of $\Gs = \Bs$ solvable. In that case, $\Ss = e +\ttt$, where $e$ is the chosen regular nilpotent, and $\ttt = \Lie(\Ts)$ for a maximal torus $\Ts$ containing $h$. We have $H^*_\Bs = H^*_\Ts = \C[\ttt]$ and we denote by $\I$ the augmentation ideal, i.e. the maximal ideal in $H^*_\Ts$ of the polynomials vanishing at $0\in\ttt$.

Notice that the semisimple elements are dense in $\Ss$. Indeed, the regular locus in $\ttt$ is a complement of hyperplanes defined by the characters of the $\Ts$-action on $\geg$. By \cite[Corollary 2.26]{HR}, if $v\in \ttt$ is regular, $e+v$ is semisimple as well -- and conjugate to $v$.

The assumption $\dim \ZZ = \dim \Ss$ implies that the generic fibres of the projection $\ZZ\to \Ss$ are finite, and hence a generic regular semisimple element has finitely many zeros. As the $\Ts$-fixed points are zeros of semisimple elements, there are only finitely many of them. If $X$ is a smooth projective variety with isolated torus fixed points, its cohomology is Tate. Hence we have proved

\begin{lemma} \label{tate}
Assume a solvable principally paired group $\Bs$ acts on a smooth projective variety $X$, and the zero scheme $\ZZ\subset \Ss\times X$ over the Kostant section $\Ss=e+\ttt$ satisfies $\dim \ZZ = \dim \Ss$. Then the cohomology of $X$ is Tate, i.e. $H^q(X,\Omega^p) = 0$ whenever $p\neq q$.
\end{lemma}

We now prove a comparison between the Poincar\'{e} series. 

\begin{lemma}
Assume a solvable principally paired group $\Bs$ acts on a smooth projective variety $X$, and the zero scheme $\ZZ\subset \Ss\times X$ satisfies $\dim \ZZ = \dim \Ss$. Then
\begin{enumerate}
    \item $H^i(\ZZ,\OO_\ZZ) = 0$ for $i>0$;
    \item $ P_{H^*_\Ts(X)}(t) = P_{\C[\ZZ]}(t).$
\end{enumerate}
\end{lemma}

\begin{proof}
As the maximal torus $\Ts\subset \Bs$ acts algebraically on the smooth projective variety $X$, we have by \cite[14.1]{GKM} that $X$ is equivariantly formal. In particular, $H^*_\Ts(X)$ is a free graded module over $H^*_\Ts$, and there is an isomorphism $H^*_\Ts(X)/ \I H^*_\Ts(X) \simeq H^*(X)$ of graded $\C$-algebras. That allows us to compare the Poincar\'{e} series of equivariant and non-equivariant cohomology:

$$P_{H^*_\Ts(X)}(t) = \frac{P_{H^*(X)}(t)}{(1-t^2)^r},$$
where $r = \dim \Ts$ is the rank of $\Bs$.

The scheme $\ZZ$ is defined as the zero scheme of the vector field $V_\Ss$, which is a section of the vertical tangent bundle $T_v$. That bundle is the pullback of $T_X$ via the projection $\pi_2:\Ss\times X\to X$.

Note that $T_v$ is a vector bundle of rank equal to $n = \dim X$. Let $\Omega_v^p = \Lambda^p T_v^*$. Then the condition $\dim \ZZ = \dim \Ss$ implies that the Koszul complex
\begin{equation}
\label{Kosbig}
0\to \Omega_v^n \xrightarrow{\iota_{V_{\Ss}}} \Omega_v^{n-1} \xrightarrow{\iota_{V_{\Ss}}} \dots \xrightarrow{\iota_{V_{\Ss}}}  \Omega_v^1  \xrightarrow{\iota_{V_{\Ss}}}  \Omega_v^0 \to 0
\end{equation}
is a resolution of $\OO_\ZZ$. Therefore the hypercohomology of this complex equals the cohomology of $\OO_\ZZ$ on $\Ss\times X$. But note that as
$$H^*(\Ss\times X,\OO_\ZZ) = H^*(\ZZ,\OO_\ZZ),$$
the cohomology ring $H^*(\ZZ,\OO_\ZZ)$ is isomorphic to the hypercohomology of the Koszul complex \eqref{Kosbig}. We denote that Koszul complex by $K^\bullet$, with cohomological grading, where the index in superscript ranges from $-n$ to $0$, so that $K^p = \Omega_v^{-p}$.
We see that there is a spectral sequence with the first page
\begin{equation}
\label{spectr}
	E_1^{pq} = H^q(\Ss\times X,\Omega_v^{-p})
\end{equation}
convergent to $H^{p+q}(\ZZ,\OO_\ZZ)$.

From Lemma \ref{tate} we know that $H^q(X,\Omega^{-p})$ vanishes whenever $p+q\neq 0$. Note that for any vector bundle $\Ee$ on $X$ we have
$$H^q(\Ss\times X,\pi_2^*(\Ee)) \cong H^q(X,\Ee) \otimes \C[\ttt].$$
Indeed, the equality clearly holds for global sections and $\C[\ttt]$ is flat over $\C$. Therefore by Lemma \ref{tate}, the only potentially nonzero entries in the first page of the spectral sequence \eqref{spectr} are
$$E_1^{-p,p} = H^p(\Ss\times X,\Omega_v^{p}).$$

First, this means that $H^i(\ZZ,\OO_\ZZ)$ may only be nonzero if $i=0$. Second, the spectral sequence degenerates on the first page. From the degeneration, there is a filtration $F_\bullet$ on $H^0(\ZZ,\OO_\ZZ)$ such that 
\begin{equation}
\label{spseq}
    F_p/F_{p-1} \cong H^p(S\times X,\Omega_v^p) \cong H^p(X,\Omega^p) \otimes \C[\ttt].
\end{equation}
We follow the idea from the proof of Theorem \ref{clthm} and construct a $\Cs$-action on the Koszul complex.  For $t\in \Cs$, let $t_p:\Omega^p_v\to \Omega^p_v$ be the pullback of forms along the map $t:S\times X\to S\times X$. Recall that on $\Ss\times X$ the torus $\Cs$ acts by
$$t\cdot (e+v,x) = \left(t^{-2}\Ad_{H^t}(e+v), H^t x\right) = (e + t^{-2} v, H^t\cdot x)$$
and this satisfies the property $t_* V_{\Ss} = t^2 V_{\Ss}$.
As in the proof of Theorem \ref{clthm}, the following diagram commutes.

$$
\begin{tikzcd}
0\arrow[r] &
\Omega^n_v \arrow[rr, "\iota_{V_{\Ss}}"] \arrow[d, "t^{2n} t^{-1}_n"]&&
\Omega^{n-1}_v \arrow[rr, "\iota_{V_{\Ss}}"] \arrow[d, "t^{2(n-1)} t^{-1}_{n-1}"] &&
\dots \arrow[r, "\iota_{V_{\Ss}}"] \arrow[d]&
\Omega^{1}_v \arrow[rr, "\iota_{V_{\Ss}}"] \arrow[d, "t^{2} t^{-1}_1"]  &&
\OO_{\Ss\times X} \arrow[r] \arrow[d,"(t^{-1})^*"] & 0\\
0\arrow[r] &
\Omega^n_v \arrow[rr, "\iota_{V_{\Ss}}"] &&
\Omega^{n-1}_v \arrow[rr, "\iota_{V_{\Ss}}"] &&
\dots \arrow[r, "\iota_{V_{\Ss}}"] &
\Omega^{1}_v \arrow[rr, "\iota_{V_{\Ss}}"] &&
\OO_{\Ss\times X} \arrow[r] & 0
\end{tikzcd}
$$

Hence we have lifted the action of $\Cs$ on $\OO_{\Ss\times X}$ to the action on the whole Koszul complex. On the level of the spectral sequence, $t\in\Cs$ acts on 
$H^p(\Ss\times X,\Omega_v^{p}) = H^p(X,\Omega^p) \otimes \C[\ttt]$ by $t^{2p} t^{-1}_{p}$. As the multiplication by $t$ on $X$ is homotopic to the identity, $t_p^{-1}$ induces the identity on $H^p(X,\Omega^p)$. Therefore $\Cs$ acts on $H^p(X,\Omega^p)$ with the weight $2p$. Additionally, the action on $\C[\ttt]$ is nontrivial -- it is a scaling of weight 2. Hence the Poincar\'{e} series of $\C[\ttt]$ is $\frac{1}{(1-t^2)^r}$. As the spectral sequence is $\Cs$-equivariant, applying \eqref{spseq} we get

$$P_{\C[\ZZ]}(t) = \frac{\sum \dim H^p(X,\Omega^p) t^{2p}}{(1-t^2)^r} = \frac{P_{H^*(X)}(t)}{(1-t^2)^r} = 
P_{H^*_\Ts(X)}(t),$$
where the last part follows from the equivariant formality.
\end{proof}

To finish the solvable case, it remains to show the isomorphism $\rho:H^*_\Ts(X)\to H^0(\ZZ,\OO_\ZZ)$. We know that the Poincar\'{e} polynomials match on both sides, so it is enough to define a natural injective homomorphism $\rho$.

\begin{lemma} \label{lemrho}
Assume that a solvable principally paired group $\Bs$ acts on a smooth projective variety $X$, and the zero scheme $\ZZ\subset \Ss\times X$ over the Kostant section $\Ss=e+\ttt$ satisfies $\dim \ZZ = \dim \Ss$. Then there exists an injective graded map
\begin{equation}\label{nonregmap}
\rho:H^*_\Ts(X)\to \C[\ZZ]^{\red}
\end{equation}
from the equivariant cohomology to the reduction of the ring of functions on $\ZZ$.
\end{lemma}

Note that a priori, the scheme $\ZZ$ could be non-reduced, and so could be $\C[\ZZ]$. However, $\C[\ZZ]^{\red}$ is also a graded ring, and it is a quotient of $\C[\ZZ]$, hence every coefficient of $P_{\C[\ZZ]^{\red}}(t)$ is less than or equal to the corresponding coefficient in $P_{\C[\ZZ]}(t)$. Then the existence of an injective map \eqref{nonregmap}, together with $P_{\C[\ZZ]}(t) = P_{H^*_\Ts(X)}(t)$, will also prove that $\C[\ZZ]$ is in fact a reduced ring. This will imply $\rho$ being an isomorphism $H^*_\Ts(X)\cong \C[\ZZ]$.

\begin{proof}[Proof of Lemma \ref{lemrho}]

As the scheme $\ZZ$ is Noetherian, the ring $\C[\ZZ]^{\red}$ is the ring of those set-theoretic functions from the closed points of $\ZZ$ to $\C$ which are given by evaluations of global regular functions in $\C[\ZZ]$. Let $c\in H^*_\Ts(X,\C)$. We will define $\rho(c)$ as a function on the closed points of $\ZZ$, and show that it comes from a regular function on $\ZZ$. Let $w\in\ttt$, $x\in X$ be such that $(e+w,x)\in \ZZ$, so that the vector field $V_{e+w}$ vanishes at $x$. 

By \cite[Theorem 2.25]{HR}, there exists $M\in \Gs$ such that $e+w = \Ad_M(w+n)$ with $[w,n] = 0$ and $n$ nilpotent. Then, as $V_{e+w}$ vanishes at $x$, by \cite[Lemma 2.6]{HR} we have $x = M y$ for some $y$, which is a zero of $w+n$. Let $P$ be an irreducible component of the zero scheme of $w+n$ such that $y\in P$. From \cite[Lemma 5.12]{HR} we know that $P$ contains a fixed point $\zeta_i$ of $\Ts$. Both $P$ and $\zeta_i$ might be non-unique. However, let us make a choice and define
\begin{equation}
\label{defrho}
\rho(c)(e+w,x) = c|_{\zeta_i}(w).
\end{equation}
Here the localisation $c|_{\zeta_i}$ is an element of $H^*_\Ts(\pt) = \C[\ttt]$ and hence a function on $\ttt$, which we apply to $w$. We have thus defined $\rho(c)$ as a function on the closed points of $\ZZ$, depending on a choice for each point.

We will show that so defined $\rho(c)$ is in $\C[\ZZ]^{\red}$ -- and it will follow from the proof that the values of $\rho(c)$ do not depend on the choices. It will also be clear that the map $\rho$ is a homomorphism of rings, as the restriction to the $\Ts$-fixed points is. For the latter, it is enough to choose $P$ and $\zeta_i$ independently of $c$, i.e. uniquely for every point $(e+w,x) \in \ZZ$.

By \cite[Lemma 3.12]{HR} the algebra $H^*_\Ts(X)$ is generated by the $\Ts$-equivariant Chern classes of $\Bs$-linearised vector bundles. Therefore it is enough to prove that if $c = c_k^\Ts(\Ee)\in H^*_\Ts(X)$ is a Chern class of a $\Bs$-linearised vector bundle $\Ee$ on $X$, the function $\rho(c)$ comes from a regular function. We will prove it by showing that
\begin{equation}
\label{rhoforclass}
    \rho(c)(e+w,x) = \Tr_{\Lambda^k \Ee_x}(\Lambda^k (e+\w)_x),
\end{equation}
as the right-hand side is defined by a regular function on $\ZZ$. Notice that as the right-hand side is independent of the choices, this will also imply that $\rho$ is well-defined.

The regularity of the right-hand side of \eqref{rhoforclass} implies that the function is constant on any projective subvariety. Let us consider a fixed $(e+w,x)\in\ZZ$ and $M$ as above, so that $e+w = \Ad_M(w+n)$ with $[w,n] = 0$ and $n$ nilpotent. Let $P$ be an irreducible component of the zero scheme of $w+n$ containing $y = M^{-1}x$, and let $\zeta_i \in X^\Ts\cap P$. As $P$, and hence also $M\cdot P$, is projective, we have
$$\Tr_{\Lambda^k \Ee_x}(\Lambda^k (e+\w)_x) = 
\Tr_{\Lambda^k \Ee_{My}}(\Lambda^k (e+\w)_{My}) =
\Tr_{\Lambda^k \Ee_{M\zeta_i}}(\Lambda^k (e+\w)_{M\zeta_i}).
$$
As $\Ee$ is $\Bs$-linearised and $e+w = \Ad_M(w+n)$, this is equal to
$$\Tr_{\Lambda^k \Ee_{\zeta_i}}(\Lambda^k (\w+n)_{\zeta_i}).$$
Further, as $\w+n$ is a Jordan decomposition, this equals
$$\Tr_{\Lambda^k \Ee_{\zeta_i}}(\Lambda^k (\w)_{\zeta_i}).$$
By \cite[(3.6)]{HR} this is exactly $c|_{\zeta_i}(w)$.

Therefore we have a well-defined map $\rho$. We only need to show that it is injective. We know that for every $w\in\ttt^{\reg}$ there exists $M\in\Bs$ such that $\Ad_M(w) = e+w$. Then for any $\zeta_i\in X^\Ts$ we have $(e+w,M\zeta_i)\in\ZZ$. By definition,
$$\rho(c)(e+w,M\zeta_i) = c|_{\zeta_i}(w).$$
As the regular elements are dense in $\ttt$, the values of $c|_{\zeta_i}$ on $\ttt^\reg$ recover $c|_{\zeta_i}$. Hence from the values of $\rho(c)$ we can recover the localisations of $c$ to all the $\Ts$-fixed points. By the injectivity of localisation for equivariantly formal spaces \cite[6.3]{GKM}, these define $c$ uniquely. Therefore the map $\rho$ is injective.
\end{proof}

This finishes the proof of Theorem \ref{nonreg} in the solvable case. It remains to deduce the general case from the same result for a Borel subgroup.

\begin{proof}[Proof of Theorem \ref{nonreg} for an arbitrary principally paired group]
Let $\Bs$ be the Borel subgroup of $\Gs$ whose Lie algebra $\bb$ contains the principal nilpotent $e$. We use \cite[Section 4.3]{HR} to obtain regular maps $A:\ttt\to\Gs$ and $\chi:\ttt\to \Ss$ such that:
\begin{enumerate}
    \item $\Ad_{A(w)}(e+w) = \chi(w)$ for any $w\in \ttt$;
    \item the map $\chi$ is $\Ws$-invariant an induces an isomorphism $t/\!\!/\Ws\to \Ss$.
\end{enumerate}

We know already that when we consider the zero scheme $\ZZ_\Bs\subset (e+\ttt)\times X$, it has no coherent cohomology and $\C[\ZZ_\Bs] = H^*_\Ts(X)$. From the first point above, the morphism $(e+w,x)\mapsto (w,A(w)x)$ gives an isomorphism between $\ZZ_\Bs$ and $\ZZ_\Bs'\subset \ttt\times X$ defined as the zero scheme of the vertical vector field which over $w\in\ttt$ equals $V_{\chi(w)}$. By definition $\ZZ_\Bs'$ is then a pullback of $\ZZ$:

$$
\begin{tikzcd}
\ZZ_\Bs' \arrow[r] \arrow[d]
& \ZZ\arrow[d]
\\
\ttt \arrow[r, "\chi"]
& \Ss 
\end{tikzcd}
$$
This means that 
$$H^*(\ZZ_\Bs',\OO_{\ZZ_\Bs'}) = H^*(\ZZ,\OO_\ZZ) \otimes_{\C[\ttt]^\Ws} \C[\ttt]$$ and on the other hand 
$$H^*(\ZZ,\OO_\ZZ) =  H^*(\ZZ_\Bs',\OO_{\ZZ_\Bs'})^\Ws,$$
where the Weyl group $\Ws$ acts on the base $\ttt$. In particular, $\ZZ$ does not have higher coherent cohomology.

We already know the isomorphism $H^*_\Ts(X)\cong  H^0(\ZZ_\Bs',\OO_{\ZZ_\Bs'})$ and moreover $H^*_\Gs(X) = H^*_\Ts(X)^\Ws$. Hence to arrive at $H^*_\Gs(X)\simeq  H^0(\ZZ,\OO_\ZZ)$, we only have to prove that the Weyl action on $\ZZ_\Bs'$ coming from the action on $\ttt$ is dual to the topological action on $H^*_\Ts(X)$. This follows the proof of \cite[Lemma 4.2]{HR}. By convention we let $\Ws$ act on the left on $\ttt$, and on the right on $\C[\ttt]$ and $H^*_\Ts(X)$.

Consider a class $c\in H^*_\Ts(X) \cong \C[\ZZ_\Bs']$. For any torus-fixed point $\zeta_i\in X^\Ts$ and $\eta\in\Ws$ we have
\begin{equation}
\label{weylact}
(\eta^*c)|_{\zeta_i} = (c|_{\eta\zeta_i})\circ\eta.
\end{equation}
Here $\eta\zeta_i\in X^\Ts$ is the torus-fixed point to which $\eta T \subset \Gs$ maps $\zeta_i$. As $c|_{\eta\zeta_i}\in H^*_\Ts(\pt)=\C[\ttt]$ is a function on $\ttt$, it makes sense to precompose it with $\eta:\ttt\to\ttt$. 

We know that the localisation to torus-fixed points is injective, hence \eqref{weylact} uniquely determines $\eta^*c$. Moreover, to compare two functions on $\ttt$, we can compare their values on an open subset of $\ttt$. We will then evaluate the localisations on the open subset $\ttt^o\subset \ttt$ consisting of regular elements $w\in\ttt$ such that $V_{e+w}$ has isolated zeros on $X$.

Let $\tau:H^*_\Ts(X)\to \C[\ZZ_\Bs']$ be the isomorphism -- composition of $\rho:H^*_\Ts(X)\xrightarrow{\simeq} \C[\ZZ_\Bs]$ and the pullback along the isomorphism $(w,A(w)x)\mapsto (e+w,x)$ between $\ZZ_\Bs'$ and $\ZZ_\Bs$. Take $\w\in\ttt^o$ and $\eta\in\Ws$ and let $M_w,M_{\eta w}\in \Bs$ be such that $\Ad_{M_w}(w) = e+w$ and $\Ad_{M_{\eta w}}(\eta w) = e+\eta w$. By \eqref{defrho} for any $c\in H^*_\Ts(X)$, $\zeta_i\in X^\Ts$ and $w\in\ttt^o$ we then have
\begin{equation}
\label{woal1}
(\eta^*c)|_{\zeta_i}(w) = \rho(\eta^*c)(e+w,M_w\zeta_i) = \tau(\eta^*c)(w,A(w)M_w\, \zeta_i).
\end{equation}
On the other hand, applying \eqref{weylact}, we get
\begin{equation}
\label{woal2}
(\eta^*c)|_{\zeta_i}(w) = c|_{\eta\zeta_i}(\eta w) = \rho(c)(e+\eta w,M_{\eta w}\, \eta\zeta_i) = \tau(c)(\eta w,A(\eta w)M_{\eta w}\eta\zeta_i).
\end{equation}
By the definition of $M_w$, $M_{\eta w}$, $A_w$, $A_{\eta w}$ we have that
$$\Ad_{M_{\eta w}^{-1}A(\eta w)^{-1}A(w)M_w}(w) = \eta w.$$
As $w$ is regular, and its centraliser is connected by \cite[Corollary 3.11]{Steinberg}, we have 
$$M_{\eta w}^{-1}A(\eta w)^{-1}A(w)M_w \in \eta \Ts.$$
This means that $A(\eta w)M_{\eta w}\, \eta\zeta_i = A(w)M_w\, \zeta_i$. Hence from \eqref{woal1}, \eqref{woal2} we get that $\tau(\eta^*c) = \tau(c) \circ \eta$, where $\eta$ is understood to act on $\ZZ_\Bs'$ by acting on $\ttt$. Hence the topological Weyl group action on $H^*_\Ts(X)\simeq \C[\ZZ_\Bs']$ agrees with the Weyl group action on the $\ttt$ factor in $\ZZ_\Bs'\subset \ttt\times X$.

Therefore we have the natural graded isomorphisms
$$H^*_\Gs(X) \cong H^*_\Ts(X)^\Ws\cong \C[\ZZ_\Bs]^\Ws \cong \C[\ZZ],$$
as desired.
\end{proof}

\begin{example} \label{exsph}
Assume that $X$ is a smooth projective \emph{spherical} variety for a reductive group $\Gs$ \cite{BrionSph}. Then it satisfies the conditions of Theorem \ref{nonreg}. Recall that $X$ being spherical means that $\Gs$ acts on $X$ and the Borel subgroup $\Bs\subset \Gs$ has an open dense orbit in $X$. If $\Gs$ is a torus, then $\Bs = \Gs$ and $X$ is simply a toric variety. In particular, all the wonderful compactifications of $G/H$ for $H = G^\sigma$, where $\sigma:\Gs\to\Gs$ is an involution, are spherical \cite{deConcProc}. Classical examples include the \emph{variety of complete collineations} \cite{Thadd,VainsencherLinn,Tyrrell}, which compactifies $\PGL_n$, and the \emph{variety of complete quadrics} \cite{deConcQuad,Tyrrell,VainsencherQuad}, which compactifies $\SL_n/\SO_n$. The renewed interest in spherical varieties in recent years \cite{sph,lect} comes from their connections to mirror symmetry.

To see that spherical varieties satisfy the conditions of Theorem \ref{nonreg}, recall that spherical $\Gs$-varieties have only finitely many $\Gs$-orbits \cite[1.5]{BrionSph}. Then the claim follows from the following lemma.

\begin{lemma}
Assume that a principally paired group $\Gs$ acts on a smooth variety $X$. Let $\Ss \subset \geg$ be the Kostant section. Let $\ZZ \subset \Ss\times X$ be the zero scheme of the vector field $V_\Ss$. Then the following are equivalent:
\begin{enumerate}
    \item $\dim \ZZ = \dim \Ss$ or $\ZZ$ is empty;
    \item There are only finitely many orbits $\Gs\cdot x$ in $X$ such that the Lie algebra $\Lie(\Stab x)$ of the stabiliser contains a regular element.
\end{enumerate}
\end{lemma}

\begin{proof}
 Let us first consider a single orbit $\OO = \Gs/\Ks$ for $\Ks \subset \Gs$ a subgroup with the Lie algebra $\Lie(\Ks) = \kek$. We study $\ZZ^\reg\subset \geg^\reg\times \OO$, the zero scheme of the vector field defined by the action, restricted to the regular elements of $\geg$. By homogeneity of $\OO$, the dimensions of all the fibres of $\ZZ^{\reg}$ over the elements of $\OO$ are the same. Hence $\dim \ZZ^\reg = \dim \OO + \dim \ZZ^\reg_1$, where $\ZZ^\reg_1$ is the fibre of $\ZZ^{\reg}$ over $[1] = [\Ks]\in \Gs/\Ks$. But that fibre is simply equal to $\geg^\reg\cap \kek$.
 
 As $\geg^\reg\subset \geg$ is open and dense, that intersection is also open in $\kek$. In case $\kek$ does not contain a regular element, the intersection is empty. Otherwise, it forms an open dense subset of $\kek$, hence of dimension equal to $\dim K$. In the first case $\ZZ^\reg$ is empty. In the other case, it is of dimension $\dim \OO + \dim \kek = \dim \Gs$. Now if instead of a single orbit we consider a larger $\Gs$-variety $X$, we notice that

 \begin{enumerate}
     \item if there are finitely many orbits in $X$ with a regular element in the Lie algebra of the stabiliser, then $\ZZ^\reg$ is  either empty, or of dimension equal to the dimension of $\Gs$;
     \item if there are infinitely many orbits in $X$ with a regular element in the Lie algebra of the stabiliser, then $\ZZ^\reg$ is of dimension larger than the dimension of $\Gs$;
 \end{enumerate}
 Now, by the properties of the Kostant section, i.e. Proposition \ref{kostprop}, 
 $$\dim \ZZ = \dim \ZZ^\reg - (\dim \Gs - \rk \Gs).$$
 Then in the first case above, either $\ZZ$ is empty, or $\dim \ZZ = \rk \Gs = \dim \Ss$. In the second case, we have $\dim \ZZ > \dim \Ss$.
\end{proof}

Notice that if $X$ is projective, then by Borel fixed-point theorem $e\in \geg$ has a zero on $X$, hence $\ZZ$ is nonempty. The scheme $\ZZ$ is in general not affine. For that reason, it is not easy to calculate the ring of functions on $\ZZ$. 
This remains a computational challenge for nontrivial spherical varieties.
\end{example}

\begin{remark}
As for any spherical variety $X$ we have $\dim \ZZ = \dim \Ss$, the generic fibre of the projection $\ZZ\to \Ss$ is finite. However, the special fibres of $\ZZ\to \Ss$ in general have positive dimension. Suppose $X$ is the variety of complete collineations, i.e. the compactification of $\PGL_n$ with the action of $\GL_n\times\GL_n$, which on $\PGL_n$ comes from the left and right multiplication, i.e.
$$(A,B)\cdot M = AMB^{-1}$$
for $A,B\in\GL_n$, $M\in\PGL_n$. Now for a regular nilpotent $e\in\gl_n$, as in Example \ref{exgl}, the element $(e,e)\in \gl_n\times\gl_n$ is a regular nilpotent in $\Lie(\GL_n\times\GL_n)$. However, notice that its zero scheme in $\PGL_n$ contains the whole (projectivisation of) the centraliser of $e$, hence is of dimension $n-1$. Therefore the fibre of $\ZZ$ over the regular nilpotent is of highly positive dimension.
\end{remark}

Notice that we showed that $\C[\ZZ]$ is a reduced ring, but we did not prove that $\ZZ$ is a reduced scheme. The following example shows that it might not be.

\begin{example}\label{exthick}
Consider a copy of the Borel subgroup $\Bs_2$ of $\SL_2$, embedded as the subgroup $\Gs\subset\GL_3$  of matrices of the form
$$\begin{pmatrix}
t^2 & u & 0 \\
0 & 1 & 0 \\
0 & 0 & t^{-2}
\end{pmatrix}.$$
Let it then act on $X=\PP^2$ by restricting the standard action of $\GL_3$. We will study the zero scheme $\ZZ\subset \Ss\times X$. Notice that in this situation
$$\Ss = \left\{ \left.
\begin{pmatrix}
v & 1 & 0 \\
0 & 0 & 0 \\
0 & 0 & -v
\end{pmatrix} \right| v\in\C  \right\}\subset \gl_3.$$
The regular nilpotent $e\in\geg$ is the matrix we obtain by plugging in $v=0$.
For any $v\in \C$ we then have
$$\begin{pmatrix}
v & 1 & 0 \\
0 & 0 & 0 \\
0 & 0 & -v
\end{pmatrix}
\begin{pmatrix}
0 \\
1 \\
0 
\end{pmatrix} =
\begin{pmatrix}
1 \\
0 \\
0 
\end{pmatrix}.$$
As the right-hand side is not proportional to $\begin{pmatrix}
0 & 1 & 0
\end{pmatrix}^T$, the point $[0:1:0]\in\PP^2$ is not a zero of any element in $\Ss$. Hence $\ZZ\subset \Ss\times (\PP^2\setminus\{[0:1:0]\})$. We have
$$\PP^2\setminus\{[0:1:0]\} = X_0 \cup X_2,$$
where $X_i\subset \PP^2 = \{[x_0:x_1:x_2]\in \PP^2 | x_i\neq 0\}$. Therefore to understand $\ZZ$ it is enough to study the zero scheme over $X_0$ and over $X_2$.

Let $X_0 = \{[1:a:b]|a,b\in\C\}$. The vector field $V_e$ defined by $e$ has the local coordinates $(-a^2,-ab)$. The vector field defined by $\diag(v,0,-v)$ has the coordinates $(-va,-2vb)$. Hence the restriction $\ZZ\cap (\Ss\times X_0) \subset \Ss\times X_0 \simeq \Spec \C[v,a,b]$ is defined by the equations
$$a^2+va =0, \quad ab+2vb =0.$$
One then sees that the closed points form three lines, with equations
\begin{equation}
\label{abv}(a=0,v=0),\quad (a=0,b=0),\quad (a=-v,b=0).
\end{equation}
The first of those is however non-reduced. Namely, localising at $b$ yields equations $v^2=0,a+2v=0$.

Over $X_2 = \{[c:d:1]|c,d\in\C]$ the scheme $\ZZ\cap (\Ss\times X_2)$ has the equations
$$d+2vc=0, \quad vd=0.$$
The closed points form two lines, with equations
$$(d=0,v=0),\quad (c=0,d=0).$$
The first one, together with the first line from \eqref{abv}, glue together to a thick $\PP^1$ embedded in $\Ss\times \PP^2$ as $\{e,[x_0:0:x_2]\}$. Localisation at $c$ gives $v^2=0,dc^{-1}+2v=0$. and shows again that the component is non-reduced.

Two views of the affine part of the zero scheme $\ZZ$, inside $\Ss\times \{[x:y:1-x]|x,y\in\C\}$, are shown in Figure \ref{figthick}. All the components are visible. The thick line is blurred out in its tangent directions, and it extends to a thick $\PP^1$ in the whole $\ZZ$.

We leave it as an exercise for the reader to directly compute that 
$$\C[\ZZ] = \C[v,a]/a(a+v)(a+2v).$$

\begin{figure}[ht!]
\begin{center}
 \subfloat{
  \includegraphics[width=7cm]{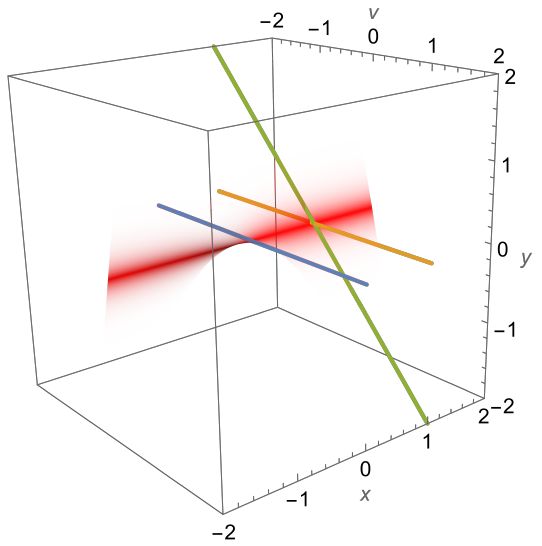}}
  \hfill
\subfloat{
  \includegraphics[width=7cm]{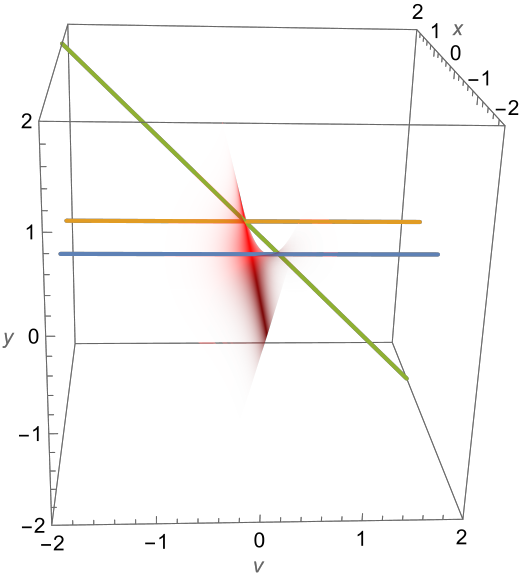}}
\end{center}
\caption{Zero scheme $\ZZ$ from Example \ref{exthick}. An affine patch $\Ss\times \{[x:y:1-x]|x,y\in\C\}$ shown.}
\label{figthick}
\end{figure}
\end{example}

\begin{remark}
In case $\Gs = \Ts$ is a torus and the action of $\Ts$ is faithful, the condition of the lemma can only be satisfied for $X$ being a toric variety. Indeed, in a torus we have $e = 0$, hence $\ZZ$ contains $\{0\}\times X$ as a subscheme. Therefore $\dim X \le \dim \Ts$. For a faithful action, this only holds when $\dim X = \dim \Ts$ and $X$ is toric. In that case, the variety is also a GKM space, so one can apply Theorem 5.16 from \cite{HR}. That theorem holds for a larger class of spaces, however it does not say anything about higher cohomology $H^i(\ZZ,\OO_\ZZ)$. In fact, if $X$ is not a toric variety, then $H^i(\ZZ,\OO_\ZZ)$ might be nonzero even for $i>0$. For example, if $X = \Gr(4,2)$ is the Grassmannian of 2-planes in $\C^4$ and $\Ts \subset \SL_4$ is a maximal torus, of dimension 3, then $H^2(\ZZ,\OO_\ZZ) = \C$.

Note also that as for a torus $\Ts$ the regular (and only) nilpotent in $\ttt = \Lie(\Ts)$ is $0$, no nontrivial variety is regular with respect to a $\Ts$-action. Therefore Theorem \ref{thmhr} cannot be applied directly.

\end{remark}

We have demonstrated that for a principally paired group $\Gs$ and a smooth projective $\Gs$-variety $X$, there are different levels of correspondence between the equivariant cohomology of $X$ and the ring of functions and coherent cohomology on $\ZZ$.

\begin{enumerate}
\item If the action of $\Gs$ is regular, then $\rho:H^*_\Gs(X)\to\C[\ZZ]$ is an isomorphism. Moreover, the variety $\ZZ$ is affine. Conversely, if $\ZZ$ is affine, then so is any fibre over an element of $\Ss$. In particular the zero scheme of $e\in\geg$ is affine. However, it is projective, as $X$ is projective, hence it is of dimension $0$. As $e$ generates an additive subgroup of $\Gs$, by \cite{Horrocks} its zero scheme has to be connected. Therefore it is just a single, potentially non-reduced, point. In other words, the action is regular.
\item If the action of $\Gs$ satisfies the condition of Theorem \ref{nonreg}, i.e. $\dim \ZZ = \rk \Gs$, then $H^*_\Gs(X)\to \C[\ZZ]$ is still an isomorphism, and the higher coherent cohomology of $\OO_\ZZ$ is trivial. However, $\ZZ$ does not need to be affine anymore.
\item There are some more situations where the action does not satisfy $\dim \ZZ = \rk \Gs$, but still $\C[\ZZ]\simeq H^*_\Gs(X)$. This covers in addition the case of a torus acting on a GKM space \cite[Theorem 5.16]{HR}. It is not clear what the exact conditions are, under which the isomorphism holds. Note that at least when the torus-fixed points are isolated, we can always define the comparison map $\rho: H^*_\Gs(X) \to \C[\ZZ]^{\red}$.
\end{enumerate}

We then have the following diagram of implications. In the first column we have conditions on the geometry of $\Gs$-action, and in the second the relations between $\ZZ$ and equivariant cohomology. In the third column one sees intrinsic properties of $\ZZ$.

$$
\begin{tikzcd}
\text{Regular action} \arrow[d, Rightarrow, thick] \arrow[r, Leftrightarrow, thick] & \ZZ\simeq \Spec H^*_\Gs(X,\C) \arrow[d, Rightarrow, thick] \arrow[r, Leftrightarrow, thick] & \ZZ \text{ affine}  \arrow[d, Rightarrow, thick] \\
\dim \ZZ = \rk \Gs \arrow[r,Rightarrow, thick] \arrow[d, Rightarrow, thick]
& H^i(\ZZ,\OO_{\ZZ}) = \begin{cases}
H^*_\Gs(X,\C) \text{ for } i = 0; \\
0 \text { otherwise}
\end{cases}  \arrow[r,Rightarrow, thick] \arrow[d, Rightarrow, thick] &
\begin{array}{c}
H^i(\ZZ,\OO_{\ZZ}) = 0\\
\text{ for } i>0
\end{array} \\
?   \arrow[r,Rightarrow, thick] & H^0(\ZZ,\OO_\ZZ) = H^*_\Gs(X,\C)
\end{tikzcd}
$$

It is not clear what should replace the question mark and work is ongoing to determine under what assumptions this result holds. In fact, L\"{o}wit \cite{jakub} proves a version of the claim for affine Bott--Samelson varieties which are not regular.

We can however clearly state the following problem.

\begin{problem}
Determine whether the implications in the second row are equivalences. That is:
\begin{enumerate}
\item does the vanishing of $H^i(\ZZ,\OO_\ZZ)$ for $i>0$ imply $H^0(\ZZ,\OO_\ZZ) = H^*_\Gs(X,\C)$;
\item does the vanishing of $H^i(\ZZ,\OO_\ZZ)$ for $i>0$ together with $H^0(\ZZ,\OO_\ZZ) = H^*_\Gs(X,\C)$ imply that $\dim \ZZ = \rk \Gs$.
\end{enumerate}
\end{problem}

\section{Singular varieties}
\label{secsing}

Theorem 5.1 in \cite{HR} shows how to recover the spectrum of equivariant cohomology of a possibly singular projective variety in a specific case. Namely, we require it to be embedded in a smooth variety with a regular action such that the restriction map on cohomology is surjective. For a general singular variety with a regular action, the situation might be much more complicated. In particular, one cannot even expect the cohomology to be concentrated in the even degrees. However, even without the surjectivity assumption, we can still infer something about a subring of the equivariant cohomology ring. Namely, we need to restrict our attention to the subring generated by the equivariant Chern classes of equivariant vector bundles. This is in some sense the subring of those elements that come from equivariant geometry. In the last 30 years, similar subrings have been considered for moduli spaces, and are known as \emph{tautological rings}. They have been introduced by Mumford in \cite{Mumford} and rose to importance with Kontsevich's work \cite{Kontsevich}. See also the expository article by Ravi Vakil \cite{Vakil}.

Let $\Gs$ be an algebraic group. For any $\Gs$-variety $X$, we define $\widetilde{H}^*_\Gs(X)$ as the subring of $H^*_\Gs(X)$ generated by the $\Gs$-equivariant Chern classes of $\Gs$-linearised vector bundles. Notice that if $X$ is a smooth variety with a regular $\Gs$-action, then by \cite[Lemma 4.14]{HR} we have $\widetilde{H}^*_\Gs(X) = H^*_\Gs(X)$.

\begin{theorem}\label{singful}
Assume that $X$ is a smooth projective variety with a regular action of a principally paired group $\Gs$. Let $Y\subset X$ be a closed $\Gs$-invariant subvariety, possibly singular. Let $\Ss$ be a Kostant section of $\Gs$. Consider the zero scheme $\ZZ^X \subset \Ss\times X$ of the vector field $V_\Ss$, as in Theorem \ref{thmhr}. Let $\ZZ^Y = \left(\ZZ^X\cap (\Ss\times Y)\right)^{\red}$ be the \textbf{reduced} intersection. Then the isomorphism $\rho:H^*_\Gs(X)\to \C[\ZZ^X]$ descends to an isomorphism $\widetilde{\rho}:\widetilde{H}^*_\Gs(Y) \to \C[\ZZ^Y]$, so that $\ZZ^Y\simeq \Spec \widetilde{H}^*_\Gs(Y)$. This makes the following diagram commute.
\begin{equation}
\label{diagsing}
\begin{tikzcd}
H^*_\Gs(X) \arrow[r, "\iota^*"] \arrow[dd,"\rho","{\mathbin{\rotatebox[origin=c]{90}{$\cong$}}}"']
& \widetilde{H}^*_\Gs(Y)  \arrow[dd,"\widetilde{\rho}","{\mathbin{\rotatebox[origin=c]{90}{$\cong$}}}"']
\\ \\
\C[\ZZ^X] \arrow[r]
& \C[\ZZ^Y]
\end{tikzcd}
\end{equation}
\end{theorem}

\begin{proof}
 As in the proof of Theorem \ref{nonreg}, the scheme $\ZZ^Y$ for a reductive or a general principally paired group is the quotient by the Weyl group action of the analogous scheme for its Borel subgroup. Similarly, if $\Ts$ is the maximal torus of $\Gs$, $\widetilde{H}^*_\Gs(X)$ is the Weyl-invariant part of $\widetilde{H}^*_\Ts(X)$, as in \cite[Lemma 4.5]{HR}. Hence we can assume that $\Gs$ is solvable. Then we actually have $H^*_\Gs(Y) = H^*_\Ts(Y)$ and $\widetilde{H}^*_\Gs(Y) = \widetilde{H}^*_\Ts(Y)$. Therefore we can in fact view $\widetilde{H}^*_\Gs(Y)\subset H^*_\Ts(Y)$ as generated by $\Ts$-equivariant Chern classes of $\Gs$-equivariant vector bundles.
 
As $X$ is smooth, by \cite[Lemma 4.14]{HR} we have that $H^*_\Ts(X)$ is already generated by the Chern classes of $\Gs$-equivariant vector bundles. Therefore the restriction $H^*_\Ts(X) \to H^*_\Ts(Y)$  actually has its image within $\widetilde{H}^*_\Ts(Y)$. This means that the restriction $\iota^*:H^*_\Ts(X)\to \widetilde{H}^*_\Ts(Y)$ is well-defined.

For any $c\in H^*_\Ts(X)$ the function $\rho(c)$ is defined on $(e+v,p)$ by localisation to a torus-fixed point in the same $\Gs$-orbit as $p$. Hence the values of $\rho(c)$ on the points of $\ZZ^Y$ are defined by localisation only to torus-fixed points in $Y$. Therefore the restriction of $\rho(-)$ to $\ZZ^Y$ factors through $\iota^*$. Notice that in this step we are using the reducedness of $\ZZ^Y$, as this allows us to recover a regular function from its values.

Therefore also $\widetilde{\rho}$ is well-defined. The diagram \eqref{diagsing} is then obviously commutative. We need to prove that $\widetilde{\rho}$ is an isomorphism. Surjectivity follows from surjectivity of the restriction $\C[\ZZ^X]\to \C[\ZZ^Y]$, as $\ZZ^Y$ is a closed subscheme of an affine scheme $\ZZ^X$.

Now we prove the injectivity of $\widetilde{\rho}$. Notice that we can push forward any $\Gs$-linearised vector bundle $\Ee$ on $Y$ to a $\Gs$-linearised coherent sheaf $\iota_*\Ee$ on $X$. Then for any $i$, the Chern class $c_i^\Ts(\iota_*\Ee) \in H^*_\Ts(Y)$ is well-defined, as $X$ is smooth. Its pullback to $H^*_\Ts(X)$ is $c_i^\Ts(\Ee)$, as $\iota^*\iota_* = \id$ for $\iota$ being a closed embedding. As $\iota_*\Ee$ is trivial outside of $Y$, it localises trivially to any torus-fixed point not in $Y$. Hence we have shown that any element $c\in\widetilde{H}^*_\Ts(Y)$ has a lift to $\hat{c}\in H^*_\Ts(X)$ which localises trivially to torus-fixed points outside of $Y$. Assume then that $\widetilde{\rho}(c) = 0$. This means that $\hat{c}$ also localises trivially to all torus-fixed points in $Y$, as $\widetilde{\rho}(c)$ is computed by localisation to those points and we can recover them from $\widetilde{\rho}(c)$ on the regular semisimple locus. Therefore $\hat{c} = 0$, as the localisation on $X$ is injective by equivariant formality and \cite[Theorem 1.6.2]{GKM}. Hence $c=0$. This means that $\widetilde{\rho}$ is injective.
\end{proof}

\begin{remark}
It is essential that one takes $\ZZ^Y$ to be the \emph{reduced} intersection, not just a scheme-theoretic one. In the appendix of \cite{CarBru} Carrell shows that for any non-A type there exists a Schubert variety for which a non-equivariant version of the result does not hold. As the cohomology of Schubert varieties is already generated by Chern classes of equivariant bundles, the above theorem will have to fail if we skip the reduction in the definition of $\ZZ^Y$.
\end{remark}

\begin{example} \label{discex}
This is an extension of Remark 5.4 from \cite{HR}. Described there is the discriminant variety $Y\subset X = \PP^3$. The projective space $\PP^3$ inherits the action of $\SL_2$ as a projectivisation of the 4-dimensional irreducible representation. Consider the action of $\Gs = \SL_2$ on it. If we view the elements of $\C^4$ as the polynomials $ax^3 + bx^2y+ cxy^2 + dy^3$, then $Y$ is defined by the equation
$$27a^2d^2 + 4ac^3 + 4b^3d - b^2c^2 - 18abcd = 0.$$
It parametrises the polynomials with multiple roots, or rather -- in homogeneous language -- the polynomials that contain a square in their decomposition into linear terms. The variety consists of two $\SL_2$ orbits. The first one is smooth, open and dense in $Y$ and it consists of polynomials with two distinct factors:
$$U = \{f^2g | f,g\in \Span_\C[x,y]  f\neq g\}.$$
The other one is closed in $Y$ and it consists of cubes of linear factors:
$$D = \{f^3 | f\in\Span_\C[x,y] \}.$$
The variety $Y$ is singular along the latter, see Figure \ref{discr}.

\begin{figure}[ht!]
\captionsetup{width=.65\linewidth}
\begin{center}
\includegraphics[width=.65\linewidth]{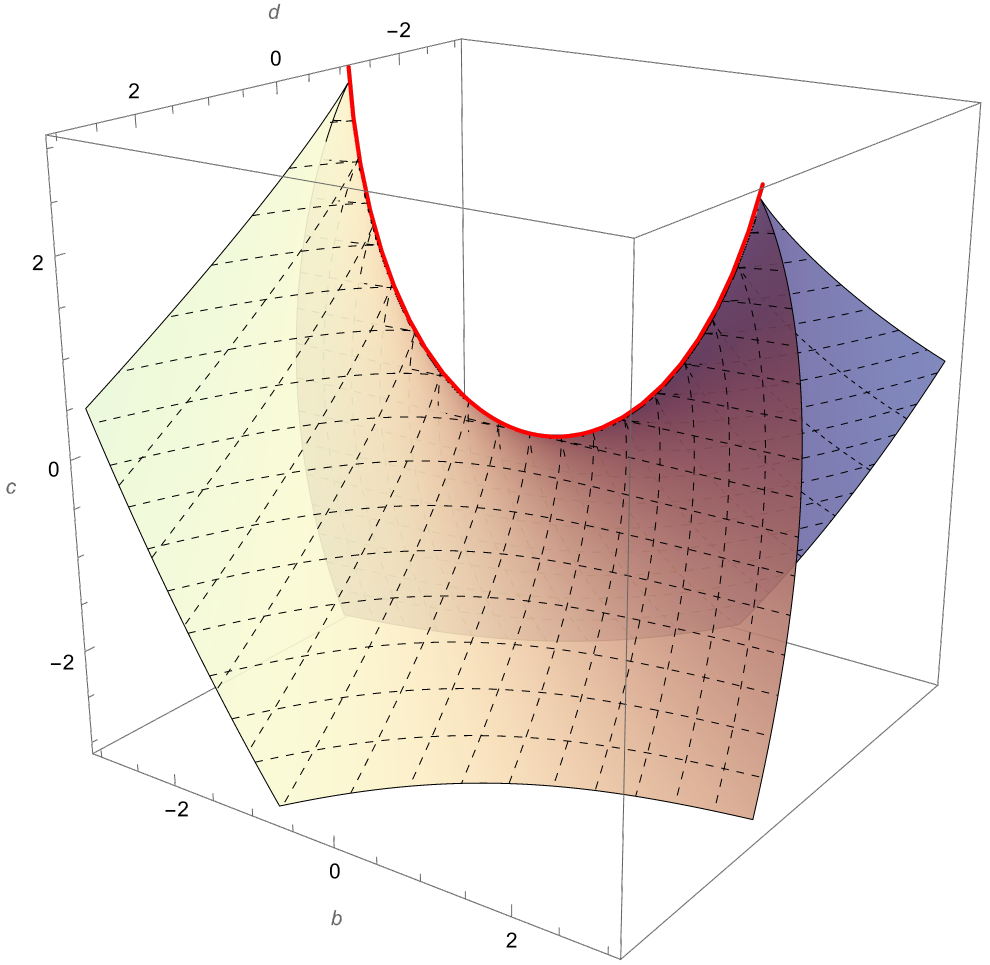}
\captionof{figure}{The discriminantal variety $Y\subset \PP^3$. Affine real part shown, for $a=1$. It is the locus of polynomials $x^3+bx^2+cx+d$ with a double root. Singular locus in red.}
\label{discr}
\end{center}
\end{figure}

However, $Y$ contains all the $\Ts$-fixed points of $X$ and so $\ZZ^X = \ZZ^Y$. Topologically $Y$ is homeomorphic to $\PP^1\times \PP^1$. Therefore $H^*_\Gs(Y)$ cannot be isomorphic to $\C[\ZZ^Y] = \C[\ZZ^X] \simeq H^*_\Gs(\PP^3)$. The homeomoprhism $\PP^1\times\PP^1\to Y$ maps
$$(f,g) \mapsto f^2g,$$
where we view elements of $\PP^1$ as homogeneous linear polynomials. We have
$$H^*(\PP^1\times \PP^1) = \C[u,v]/(u^2,v^2)$$
with $\deg u = \deg v = 2$. Hence $P_{H^*(Y)}(t) = (1+t^2)^2$. From the equivariant formality we then have
$$P_{H^*_{\Gs}(Y)}(t) = P_{H^*(Y)}(t) \cdot P_{H^*_\Gs}(t) = \frac{(1+t^2)^2}{(1-t^4)} = 1 + 2(t^2 + t^4 + t^6 + \dots).$$
On the other hand, 
$$P_{\widetilde{H}^*_{\Gs}(Y)}(t) = P_{H^*_{\Gs}(X)}(t) = \frac{1+t^2+t^4+t^6}{(1-t^4)} = 1 + t^2 + 2(t^4 + t^6 + \dots).$$

Notice that here the restriction on equivariant cohomology is not surjective, but injective -- however, only in equivariant cohomology. As we can see from the Poincar\'{e} series, in $\widetilde{H}^*_{\Gs}(Y)$ there is one missing dimension in $H^2$, in comparison with $H^*_\Gs(X)$. In fact, by the description of the isomorphism $\PP^1\times\PP^1\to Y$, one sees that the image of $\widetilde{H}^2_\Gs(Y)$ in $H_\Gs^2(\PP^1\times\PP^1)$ is spanned by $2\hat{u} + \hat{v}$, where $\hat{u}$ and $\hat{v}$ are the generators of the two copies of $H^*_\Gs(\PP^1)$.

Notice that the analogous example can be constructed in $\PP^n$ for any $n\ge 3$. If we let $Y_n\subset \PP^n$ be the zero locus of the discriminant, then exactly the same reasoning as above will tell us that
$$\widetilde{H}^*_{\Gs}(Y_n) \simeq H^*_\Gs(\PP^n).$$
However, the full cohomology ring ${H}^*_{\Gs}(Y_n)$ is not that easy to compute anymore. We always have a surjective birational morphism $\PP^1\times \PP^{n-2} \to Y_n$ defined as above, via $(f,g)\mapsto (f^2g)$ for $f\in\PP^1$, $g\in\PP^{n-2}$. However, the map is in general not an injection.
\end{example}

\begin{remark}
 Theorem \ref{singful} requires the group $\Gs$ to act regularly on the ambient space $X$ -- and hence also on $Y$. However, in Theorem \ref{nonreg} we proved a version for a smooth $\Gs$-variety, where regularity was not necessary. It is then an obvious question whether there is a common generalisation, that would work for non-regular, singular varieties. It is not yet clear if that is possible or not. An important part of Theorem \ref{singful} is that the scheme $\ZZ^X$ is reduced. However, in a non-regular case, the scheme might be non-reduced, as in Example \ref{exthick}, and hence defining $\ZZ^Y$ as a reduced intersection cannot work in general. Combining both results is an interesting challenge that is yet to be resolved.
\end{remark}

\begin{remark}
A related work of Hausel \cite{bigalg} studies the equivariant cohomology ring, as well as equivariant intersection cohomology, of affine Schubert varieties. They are realised as rings of functions of zero schemes of commuting vector fields on a projective space. The definition of the vector fields is more involved.
\end{remark}

\section{K-theory conjecture for GKM spaces} \label{seck}
In \cite{HR} a question was posed about extensions of Theorem \ref{thmhr} to other cohomology theories. The first non-trivial one, different from equivariant cohomology, is equivariant K-theory. Recall that the equivariant cohomology $H^*_\Gs(\pt)$ of the point is equal to $\C[\ttt]^\Ws = \C[\geg]^\Gs$. On the other hand, the equivariant algebraic K-theory $K^0_\Gs(\pt)$ of the point is equal to the representation ring $R(\Gs)$ of $\Gs$. After tensoring with $\C$, this becomes isomorphic to $\C[\Ts]^\Ws = \C[\Gs]^\Gs$ \cite[Theorem 6.1.4]{ChGi}.

Hence, to switch from equivariant cohomology to equivariant K-theory, we switch from the $\Gs$-invariant functions on the Lie algebra, to the $\Gs$-invariant functions on the group. This suggests that in line with the previous results, one could perhaps recover K-theory by considering the action of the group instead of the Lie algebra. To this end, for an action of $\Gs$ on a scheme $X$, define the \emph{fixed point scheme} $\Fix_\Gs(X)$ by the following pullback diagram, cf. \cite[(1.1.1)]{Low}.

$$ \begin{tikzcd}
  \Fix_\Gs(X) \arrow{r}\arrow{d} &
  \Gs\times X\arrow{d}{\rho\times \pi_2}  \\
  X \arrow{r}{\Delta} &
  X\times X.
\end{tikzcd}$$

Here $X$ maps to $X\times X$ diagonally and the map $\rho\times \pi_2$ maps $(g,x)$ to $(gx,x)$. Therefore, the $\C$-points of $\Fix_\Gs(X)$ parametrise the pairs $(g,x)\in \Gs\times X$, where $gx = x$. Hausel \cite{Tamas} states the following conjecture, based on Theorem \ref{thmhr}.

\begin{conjecture}\label{kthe}
 Assume that a principally paired group $\Gs$ acts on a smooth projective variety regularly. Let $\Gs$ act on $\Fix_\Gs(X)$ by $g\cdot (h,x) = (ghg^{-1},gx)$. Then the ring $\C[\Fix_\Gs(X)]^\Gs$ of the $\Gs$-invariant functions on $\Fix_\Gs(X)$, as an algebra over $\C[\Gs]^\Gs\cong K^0_{\Gs}(\pt)\otimes \C$, is isomorphic to the equivariant algebraic K-theory $K^0_{\Gs}(X)\otimes \C$.
 $$ \begin{tikzcd}
  \C[\Fix_\Gs(X)]^\Gs \arrow{r}{\cong} &
  K_\Gs^0(X)\otimes \C  \\
  \C[\Gs]^\Gs \arrow{r}{\cong} \arrow{u}&
  K_\Gs^0(\pt)\otimes \C \arrow{u}.
\end{tikzcd}$$
\end{conjecture}
 
One can then additionally ask three questions. First, whether we can also formulate and prove the ``non-total'' version of the conjecture, where the total fixed point scheme is replaced with the fixed point scheme over some part of $\Gs$. It seems that an appropriate equivalent of the Kostant section is the Steinberg section \cite[4.15]{HumConj}, therefore we would like to prove also the analogue of Theorem \ref{thmhr}, with the Steinberg section replacing the Kostant section. If $\Gs$ a semisimple simply connected group, Holmes \cite{Holmes} and L\"{o}wit \cite[Example 1.38]{Low} have proved that the fixed point scheme for any partial flag variety $\Gs/\Ps$ over the Steinberg section is the spectrum of the equivariant K-theory.

Additionally, we are given the equivariant Chern character map, from equivariant K-theory to cohomology. This means that on the scheme level, we should get an analogous map in the other direction. Note that as the Chern character maps to the completion of cohomology, it will not give rise to an algebraic map. It has to be rather considered as a map from a formal scheme or a complex analytic map. In fact, we expect that the Chern character yields the $\Gs$-equivariant map (``geometric Chern character'')
\begin{equation} \label{cherngeo}
    \operatorname{ch_G}:(v,x) \mapsto (\exp(v), x).
\end{equation}
from $\ZZ_\tot$ to $\Fix_\Gs(X)$. Therefore along with proving Conjecture \ref{kthe}, we would like to see the Chern character geometrically. 

Third, one might ask again how important the regularity assumption is. This is elaborated on in Remark \ref{remlow} below. In this section, we prove a K-theory analogue of the GKM theorem for cohomology \cite[Theorem 5.16]{HR}.

\begin{theorem}\label{gkmk}
Let a torus $\Ts \cong (\Cs)^r$ act on a smooth projective complex variety $X$ with finitely many zero- and one-dimensional orbits. Let $\Fix'_\Ts(X) = \left(\Fix_\Ts(X)\right)^{\red}$ be the closed subscheme of $\Fix_\Ts(X)$ given by reduction. Then $\C[\Fix'_\Ts(X)] \cong K^0_\Ts(X)\otimes \C$ as algebras over $\C[\Ts]\simeq K^0_\Ts(\pt)\otimes \C$:
$$ \begin{tikzcd}
  \C[\Fix'_\Ts(X)] \arrow{r}{\cong}  &
  K^0_\Ts(X)\otimes \C   \\
   \C[\Ts] \arrow{r}{\cong} \arrow{u}&
  K^0_\Ts(\pt)\otimes \C \arrow{u}.
\end{tikzcd}$$
\end{theorem}

\begin{proof}
We proceed as in the proof of \cite[Theorem 5.16]{HR}. We first construct a map
$$\rho: K^0_\Ts(X)\otimes \C \to \C[\Fix'_\Ts(X)]$$
and then an injective left inverse.

FOr any $x\in K^0_\Ts(X)$ we can define $\rho(x)$ by its values, as $\Fix'_\Ts(X)$ is reduced. For any $\Ts$-linearised vector bundle $\Ee$ on $X$ let
$$\rho(\Ee)(t,p) = \Tr_{\Ee_p}(t).$$
Note that if $0\to \Ee \to \F \to \Geg \to 0$ is an exact sequence of $\Ts$-linearised vector bundles, then for any $(t,p)\in \Fix'_\Ts(X)$ we have $\Tr_{\F_p}(t) = \Tr_{\Ee_p}(t) + \Tr_{\Geg_p}(t)$. Therefore the function $\rho$ is well defined on $K^0_\Ts(X)$, and by definition it always gives a regular function on $\Fix'_\Ts(X)$.

Let us denote the $\Ts$-fixed points by $\zeta_1$, $\zeta_2$, \dots, $\zeta_s$ and the one-dimensional orbits by $E_1$, $E_2$, \dots, $E_{\ell}$. The closure of any $E_i$ is an embedding of $\PP^1$ and contains two $\Ts$-fixed points $\zeta_{i_0}$ and $\zeta_{i_\infty}$, which for any $x\in E_i$ are equal to the limits $\lim_{t\to 0} t x$ and $\lim_{t\to\infty} tx$. The action of $\Ts$ on $E_i$ has a kernel $K_i$ of codimension $1$. Then by \cite[Corollary A.5]{Rosu} the localisation $K_\Ts(X)\otimes \C\to K_\Ts(X^\Ts)\otimes \C$ is injective and its image is

$$
H = \left\{
(f_1,f_2,\dots,f_s)\in \C[\Ts]^s  \, \bigg| \, f_{i_0}|_{K_i} = f_{i_\infty}|_{K_i} \text{ for } j=1,2,\dots,\ell
\right\}.
$$
We will use the following lemma \cite[Lemma 5.10]{HR}.

\begin{lemma}\label{projinj}
 Let $Y$ be a reduced scheme over a field $k$. Assume that $Z$ is a closed subvariety and every closed point $p\in Y$ is contained in a projective connected subvariety that intersects $Z$. Then the restriction map on regular functions $k[Y]\to k[Z]$ is injective.
\end{lemma}

Take any $(t,p)\in \Fix'_\Ts(X)$. Then $p$ lies in the fixed point scheme of $t\in\Ts$. As $\Ts$ is commutative, any translate of $p$ is also a fixed point of $t$. Therefore we have $\{t\}\times \overline{T\cdot p}\subset \Fix'_\Ts(X)$ as a closed projective subvariety. Then by the Borel fixed point theorem it contains a fixed point of the whole torus. Hence the conditions of Lemma \ref{projinj} are satisfied for the inclusion $\Ts\times X^\Ts \subset \Fix_\Ts(X)$, so the restriction of global functions $\C[\Fix'_\Ts(X)]\to\C[\Ts\times X^\Ts]$ is injective.

From this, we construct an injective left inverse of $\rho$. First, note that $\C[\Ts\times X^\Ts] = \C[\Ts]^s$. Then notice that the restriction $\tau:\C[\Fix'_\Ts(X)]\to\C[\Ts\times X^\Ts] = \C[\Ts]^s$ actually maps into $H$. Indeed, take $i\in\{1,2,\dots,\ell\}$. We want to prove that for any $f\in\C[\Fix'_\Ts(X)]$ the functions $\tau(f)|_{\Ts\times \zeta_{i_0}}$ and $\tau(f)|_{\Ts\times \zeta_{i_\infty}}$ are equal as functions on $\Ts$.

This follows as in the fixed point scheme, for any $t\in K_i$ both points $(t,\zeta_{i_0})$ and $(t,\zeta_{i_\infty})$ lie in the same connected projective subvariety $t\times \overline{E_i}$. Therefore the values of $f$ on those two points are the same.

Having proved that $\tau$ is an injective map to $H\simeq K_\Ts(X)\otimes\C$, we only need to prove that $\tau\circ\rho = \id$. Take a class $[\Ee]\in K^0_\Ts(X)$, where $\Ee$ is a $\Ts$-linearised vector bundle on $X$. Then $\tau\circ\rho([\Ee])$ is a function on $\Ts\times X^\Ts$ which at the point $(t,\zeta_i)$ attains the value $\Tr_{\Ee_{\zeta_i}}(t)$. When we fix $\zeta_i$, this means that we restrict $\Ee$ to $\zeta_i$ to get a representation of $\Ts$, and then we compute the trace of $t$. But this is exactly how we define the isomorphism $K_0^*(\pt) \cong \C[\Ts]$. Hence $\tau\circ\rho([\Ee])$, as an element on $H\subset \C[\Ts]^s$ is exactly the localisation of $\Ee$ to $X^\Ts$ -- and that localisation is the given isomorphism between $K_\Ts(X)\otimes \C\cong H$.
\end{proof}

\begin{remark}
The zeros of a torus action on a smooth variety are reduced, cf. \cite[13.1]{Milne}. Therefore the author also expects the scheme $\Fix_\Ts(X)$ to actually be reduced in this situation, so that $\Fix_\Ts(X) = \Fix'_\Ts(X)$. However, the exact argument is missing.
\end{remark}

\begin{remark} \label{remlow}
A big advantage of working with K-theory is that it generalises easily to arbitrary base fields. The situation is much more complicated for cohomology, where over finite fields there are many different competing cohomology theories. Therefore Conjecture \ref{kthe} can be stated over any field. In fact, L\"{o}wit \cite{Low} proves a (derived) version for toric varieties and affine Schubert varieties over finite fields. Notice that in both cases, the $\Gs$-varieties are usually not regular.
\end{remark}

\begin{remark}
    In certain situations, we have constructed the ring homomorphisms $\rho_H:H^*_\Gs(X)\to \C[\ZZ_\tot^{\red}]$ as in Theorem \ref{thmhr} and $\rho_K:K^0_\Gs(X)\otimes \C\to \C[\Fix'_\Gs(X)]$ as in Theorem \ref{gkmk}. Whenever they are both well-defined, the complex analytic map $\ch_\Gs$ $\eqref{cherngeo}$ is compatible with the Chern character $\ch:K^0_\Gs(X)\otimes \C \to H^*_\Gs(X)$. This means that the following diagram commutes.
    $$
    \begin{tikzcd}
        K^0_\Gs(X)\otimes\C\arrow{d}{\rho_K}  \arrow{r}{\ch}  & H^*_\Gs(X)\arrow{d}{\rho_H}\\
        H^*_\Gs(X) \arrow{r}{\ch_G^*} & \C[\ZZ_\tot^{\red}]
    \end{tikzcd}
    $$
This is a simple consequence of the following two facts:
    \begin{enumerate}
        \item For any $\Gs$-linearised vector bundle $\Ee$ and $(t,p)\in \Fix_\Gs(X)$ we have $\rho_K(\Ee)(t,p) = \Tr_{\Ee_p}(t)$;
        \item For any $\Gs$-linearised vector bundle $\Ee$, a nonnegative integer $k$, and $(v,p)\in \ZZ_\tot(X)$ we have
        $\rho_H(c_k(\Ee)) = \Tr_{\Lambda^k \Ee_p}(\Lambda^k (v)_x)$
    \end{enumerate}
and of the construction of the Chern character from Chern roots.
\end{remark}

\end{document}